\documentclass[12pt]{amsart}
\usepackage{amsmath}
\usepackage{amssymb}
\usepackage{tabularx}
\usepackage{enumerate}
\usepackage{graphicx}
\usepackage{texdraw}
\usepackage{color}

\usepackage{pgfplots}
\usepackage{graphics}
\usepackage{graphicx}
\usepackage{subfigure}
\usepackage{mathrsfs}

\usepackage{amsfonts,amssymb,amsmath}
\usepackage{epsfig}
\usepackage{bm}

\newtheorem{Theorem}{Theorem}[section]
\newtheorem{Lemma}[Theorem]{Lemma}
\newtheorem{Proposition}[Theorem]{Proposition}

\newtheorem{Remark}[Theorem]{Remark}

\newtheorem{example}{Example}[section]
\newtheorem{Definition}{Definition}[section]

\numberwithin{equation}{section}
\numberwithin{figure}{section}

\begin{document}

\title[]{Monotone methods for semilinear parabolic  and elliptic equations on graphs}

\author[Y. Hu]{Yuanyang Hu$^1$}
\author[C. Lei]{Chengxia Lei$^2$}
\thanks{$^1$ School of Mathematics and Statistics,
	Henan University, Kaifeng, Henan 475004, P. R. China.}
\thanks{$^2$School of Mathematics and Statistics, Jiangsu Normal University,
	Xuzhou, 221116, Jiangsu Province, China.}

\thanks{{\bf Emails:} {\sf yuanyhu@mail.ustc.edu.cn} (Y. Hu).}
\thanks{{\bf Emails:} {\sf leichengxia001@163.com(C. Lei)}}
\thanks{Y. Hu was partially supported by National Natural Science Foundation of
	He' nan Province of China (Grant No. 222300420416) and China Postdoctoral Science Foundation (No. 2022M711045).}
\thanks{C. Lei was partially supported by NSF of China (No.11801232,11971454), the NSF of Jiangsu Province (No.BK20180999) and the Foundation of Jiangsu Normal University (No. 17XLR008).}
\date{\today}

\begin{abstract}
		This paper is devoted to investigate the extinction and propagation properties of solutions to the graph Laplacian parabolic problems with Kpp type or Allen-Cahn type forcing terms on graphs. To this end, we establish the (strong) maximum principle and the upper and lower solutions method for parabolic and elliptic problems on graphs. The stability of equilibrium solutions is studied by constructing suitable upper and lower solutions. Moreover, we give an example and numerical experiments to demonstrate one of our main results. 
\end{abstract}

\keywords{ Graph Laplacian, Maximum principle, Extinction and propagation of solutions, Stability}

\maketitle

\section{Introduction}
The evolution of species has been paied much attention by biologists during the rencent decades, see \cite{CC} and references therein. Numerous mathematical models have been proposed to study the spreading of the species through  the random population diffusion. One of the models based on the reaction-diffusion equation over the entire space  
\begin{equation}\label{fk}
	\begin{cases}
		u_{t}=\Delta u+f(u)~\text{in}~\mathbb{R}^{n}\times\mathbb{R}^{+},\\
		u(x,0)=u_0 (x)~\text{in}~\mathbb{R}^{n}\times\{ 0\},
	\end{cases}
\end{equation}
has been extensively discussed, please see \cite{AW,AWa,F,KPP,LZ} and references therein. Here the Laplacian operator $\Delta $ is used to describe the mobility of species, where the probability of moving in all directions is the same, and the nonlinear function $f(u)$ satisfies
\begin{equation}\label{11}
	f\in C^{1}[0,1],~f(0)=f(1)=0.
\end{equation}
Furthermore, either
\begin{equation}\label{52}
	f^{'}(0)>0,~f(u)>0~\text{for}~u\in(0,1),
\end{equation}
or 
\begin{equation}\label{b}
	\begin{cases}
		\text{there exists}~ \alpha\in (0,1)~ \text{such that}~f(u)<0~\text{for all}~u\in (0,\alpha),\\~f(u)>0~\text{for all}~u\in(\alpha,1),
		~f^{'}(0)<0,~\int_{0}^{1} f(u) du> 0.
	\end{cases}
\end{equation}
Following common jargon, if the nonlinear function $f(u)$ satisfies \eqref{11} and \eqref{52}, then we call that $f$ is a KPP type anf if $f$ satisfies \eqref{11} and \eqref{b}, then $f$ is said to be of Allen-Cahn type. 

The travelling waves and spreading speeds of Cauchy problems have been discussed by many authors, see \cite{FZ,W1,YZ} for more details, which are used to describe the propagation of species. Specifically, Fisher \cite{F} and Kolmogorov, Petrovsky and Piskunov \cite{KPP} studied the extinction and propagation of solutions and the spreading speed of Cauchy problem \eqref{fk}  for one dimensional space ($n=1$). Then those conclusions were extended to high dimension for Cauchy problem \eqref{fk} by Aronson and Weinberger \cite{AWa}. In particular, if $u_0$ has compact support, $f$ satisfies \eqref{52} and
$$
	\lim\limits_{u\to 0^{+}} \frac{f(u)}{u^{1+\frac{2}{n}}} >0,
$$	
then the authors (see \cite{AWa} for details) came to the conclusion that there exists $c^{*}>0$ such that 
\begin{equation*}
	\lim\limits_{t\to +\infty} \sup\limits_{0\le x< ct} \left| u(x,t)-1 \right| =0 ~\text{for any}~c\in (0,c^{*}), 
\end{equation*}
\begin{equation*}
	\lim\limits_{t\to +\infty} \sup\limits_{ ct\le x } \left| u(x,t) \right| =0 ~\text{for any}~c> c^{*}. 
\end{equation*}
Naturally, a similar result holds for $x\le 0$; if $f$ satisfies $\liminf\limits_{t\to +\infty}u(x,t) \ge l$ uniformly on all compact subset of $\mathbb{R}^{n}$ for some $l\in (0,1]$, then for any $c\in(0,c^{*})$ and any $y\in \mathbb{R}^{n}$ $\lim\limits_{t\to +\infty}\min\limits_{| x-y|\le ct}u(x,t) \ge l$.

Recently, Matano, Punzo and Tesei \cite{MPT} considered the Cauchy problem with the Laplace-Beltrami operator 
 \begin{equation}\label{h}
	\begin{cases}
		u_{t}-\Delta_{H} u= f(u)~\text{in}~H^{n}\times(0,\infty),\\
		u=u_{0}(x)~\text{in}~H^{n}\times\{0\}	
	\end{cases}
\end{equation} 
on the hyperbolic space $H^{n}(n\ge 2)$. The study of Laplace-Beltrami operator has attracted much attention for long time, see \cite{B,G,NS} for more details. 

  Matano et. al. \cite{MPT} prove results for problem \eqref{h} which are analogous to those for problem \eqref{fk}. Furthermore, they obtained some new interesting conclusions. Let us recall the two new results of \eqref{h} in \cite{MPT}.

(i) if $f$ satisfies \eqref{11}, \eqref{52}
and $\sup\limits_{u\in (0,1]} \frac{f(u)}{u}=f^{'}(0)$, then extinction prevails provided that $c_0< n-1$ and $u_0(x)$ has compact support, whereas there is propagation if $c_0> n-1$, where $c_0 =2 \sqrt{f^{'}(0)}$. However, the propagation of solutions always happen (``hair-trigger effect'') for the Euclidean case \eqref{fk} if $f$ is a KPP type. 

(ii) Suppose that $f$ satisfies \eqref{11} and \eqref{b}. Then extinction occurs if the initial data function $u_0$ is sufficiently small, respectively propagation happens if $u_0$ is sufficiently large and $c_0> n-1$.

The classical Laplacian operator $\Delta$ and the Laplace-Beltrami operator $\Delta_{H}$  are arrived by the assumption that the probability of species moving in all directions is the same. But some species move towards resource abundance, for example, migratory bird migrate for food in the winter \cite{LY}. 
Let $G=G(V,E)$ and $\bar{\Omega}=\Omega\cup \partial \Omega $ be finite connected weighted graphs (see section 2). In the light of the above factor, we consider the following discrete Cauchy problem  
 \begin{equation}\label{1pb}
	\begin{cases}
		u_{t}-\Delta_{{V}} u= f(u),~&x\in V,~t>0, \\
		u(x,0)=u_0(x),~&x\in V.
	\end{cases}
\end{equation}
on $G$ to describe the anisotropic diffusion of species, where $f$ is a Kpp type or an Allen-Cahn type and $u_0 :V \to \mathbb{R}^{+}$ is a given function.

There are many reaction-diffusion equations with initial-boundary values on bounded domains in Eulicd space used to describe the spread of species, see \cite{CC} and references therein. But in the real world, sptial discrete model is more reasonable in explaining some ecological phenomena. (for example, see e.g. \cite{LT}). Thus, we also study the following parabolic and elliptic boundary value problem
\begin{equation}\label{1p1}
	\begin{cases}
		u_{t}-\Delta_{{\Omega}} u= u(a-b u),~&x\in \Omega,~t>0, \\
		\mathcal{B}u=0,~&x\in \partial{\Omega},~t>0,\\
		u(x,0)=u_0(x),~&x\in \Omega,
	\end{cases}
\end{equation}
where $\mathcal{B}u=u$ or $\mathcal{B}u=\frac{\partial u}{\partial{\omega} n}$ (see \eqref{23}), $\Delta_{\Omega}$ is the usual graph Laplacian on $\Omega$ (see \eqref{24}) and $u_0 :\Omega\to \mathbb{R}^{+}$ is a given function.

Both \eqref{1pb} and \eqref{1p1} can be regarded as the discrete version of \eqref{fk}.
Recently, increasing efforts have been devoted to the development of analysis on graphs. In \cite{CB}, Chung and Berenstein  studied the inverse conductivity problem and proved the solvability of the Dirichlet boundary value problem and the Neumann boundary value problem on finite graphs. Lin and Wu \cite{LW} established the existence and nonexistence of global solutions for the semilinear heat equation on graphs. Huang, Lin and Yau studied the mean field equation on graphs, and they used the upper and lower solutions method to prove an existence result \cite[Theorem 2.2]{HLY}, which is consist with the conclusion of Caffarelli and Yang \cite{CY} on doubly periodic regions in $\mathbb{R}^2$. In a weighted network SIRS epidemic model, Liu and Tian \cite{LT} applied the upper and lower solutions approach to show that the disease-free equilibrium is asymptotically stable if the basic reproduction number is lower than one.  In \cite{KC1}, Kim and Chung established a comparison principle for the $p$-Laplacian on networks. As an effective method to study PDEs, the upper and lower method has been attracted much attention from a lot of researchers for many years, see \cite{BDJ,KC,P,S,W} and references therein. Motivated by the above literature, we develop the lower and upper solutions method for the following nonlinear discrete parabolic problem
with initial condition
\begin{equation}\label{34}
	\begin{cases}
		u_{t}-{\Delta_{{V}}} u=f(x,t,u),~x\in V,~t>0,\\
		u(x,0)=u_0(x)~,x\in V,
	\end{cases}
\end{equation}
 and the discrete elliptic equation
\begin{equation}\label{tc}
	-\Delta_{{V}} u-\vec{b}\cdot\nabla u+c(x)u= f(x,u),~x\in V,
\end{equation} on $G$, and
the discrete parabolic initial-boundary value problem
\begin{equation}\label{1s}
	\begin{cases}
		u_t-\Delta_{\Omega} u= f(x,t,u),~&(x,t)\in \Omega\times (0,T], \\
		\mathcal{B}u= g,~&(x,t)\in \partial \Omega\times (0,T],\\
		u(x,0)= \phi,~&x\in \Omega,		
	\end{cases}	
\end{equation}
and the following discrete elliptic boundary value problem
\begin{equation}\label{t}
	\begin{cases}
		-\Delta_{{\Omega}} u= f(x,u),~&x\in \Omega,  \\
		\mathcal{B} u= \phi(x),~&x\in \partial{\Omega},
	\end{cases}
\end{equation}
 on $\bar{\Omega}$, where $\mathcal{B}u=u$ or $\mathcal{B}u=\frac{\partial u}{\partial{\omega} n}$ (see \eqref{23}), $\Delta_{\Omega}$ is the usual graph Laplacian on $\Omega$ (see \eqref{24}) and $\phi:\Omega\to \mathbb{R}^{+}$ is a given function.

 By  the upper and lower solutions method, we obtain the existence and uniqueness of nonnegative global solutions to \eqref{1pb} and \eqref{1p1} and give the long time behavior of  solutions to \eqref{1pb} and \eqref{1p1}. We prove results for problem \eqref{1pb} which are analogous to those above for problems \eqref{fk} and \eqref{h}, yet exhibit remarkably novel features compared to the Euclidean case and the Hyperbolic case.

The rest of the paper is arranged as below. In section 2, we introduce the preliminary concepts on graphs and some properties of eigenvalues and eigenfunctions of the graph Laplacian operetors. In section 3, we establish the maximum principle for elliptic and parabolic problems on finite graphs. To this end, we prove the existence and uniqueness of solutions to a class of linear parabolic and elliptic equations on graphs. The lower and upper solutions method for some elliptic and parabolic problems is investigated in section 4.
In section 5, the monotonicity and convergence of solutions of some initial-boundary value and initial value problems on graphs are discussed. And a lot of examples are given to show how the method of upper and lower solutions can be employed.  In particular, we investigate the problems \eqref{1pb} and \eqref{1p1} and describe the long time dynamic behavior of solutions to these problems.

\section{Preliminaries}
\subsection{Some definitions on graphs}
 (see \cite{CB} for details). 
A graph is represention of a set of objects, called vertices, where some pairs of vertices are connected by links, which is called edges, and is denoted by $G=G(V,E)$ where $V$ is the set of vertices and $E$ is the set of edges, that is, $E$ consists of some couples $(x,y)$ where $x,y\in V$. We write $x\sim y$ ($x$ is connected to $y$, or $x$ is joint to $y$, or $x$ is adjacent to $y$, or $x$ is a neighbor of $y$) if $(x,y)\in E$. We denote the edge $(x,y)$ by $\overline{xy}$, and call $x$, $y$ are the endpoints of this edge. The edge $\overline{xy}$ is called a loop if it has the same endpoints (should it exist), i.e., $x=y$. A graph is called simple if it has neither loops or multiple edges.  

A graph is called undirected or unoriented, if the couples $(x,y)$ are unordered, that is, $(x,y)=(y,x)$.

For a notational convenience, we write either $x\in G$ or $x\in V$ if $x$ is a vertex in $G(V,E)$.

If the number of vertices of a graph $G$ is finite, then we say $G$ is finite.

A finite sequence $\{x_{k} \}_{k=0}^{n}$ of vertices on a graph is called a path if $x_{k}\sim x_{k+1}$ for all $k=0,1,\dots,n-1$. A graph $G=(V,E)$ is said connected if, for any two vertices $x,y\in V$, there exists a path connecting $x$ and $y$, that is, a path $\{x_{k}\}_{k=0}^{n}$ such that $x_0=x$ and $x_n= y$.

A graph $\Omega^{}=\Omega^{}(V^{'},E^{'})$ is called to be a subgraph of $G=(V,E)$ if $V^{'}\subset V$ and $E^{'}\subset E$. Then, we call $G$ a host graph of $\Omega$. If $E^{'}$ consists of all the edges from $E$ which connect the vertices of $V^{'}$ in its host graph $G$, then $\Omega$ is called an induced subgraph.

A weighted graph is a graph $G=(V,E)$ associated with a weight function $\omega:~V\times V\to [0,+\infty) $ satisfying

(i) $\omega_{x y}=\omega_{yx }$, $x,y\in V$;

(ii) $\omega_{x y}=0$ if and only if $\overline{xy}\not\in E$. 

 For a subgraph $\Omega$ of a graph $G=G(V,E)$, the (vertex) boundary $\partial{\Omega}$ of $\Omega$ is the set of all vertices $z\in V$ not in $\Omega$ but adjacent to some vertex in $\Omega$, i.e., 
\begin{equation*}
	\partial{\Omega}:=\{z\in V \backslash \Omega | z\sim y ~\text{for some}~y\in \Omega \}.
\end{equation*} 
 
 Denote $\bar{\Omega}$ a graph whose vertices and edges are in $\Omega$ and vertices in $\partial{\Omega}$.
 
 {\bfseries Unless otherwise specified, all the graphs in our concern will be finite, undirected and connected, all the subgraphs in our concern are supposed to be induced, simple, undirected and connected subgraphs of a weighted graph, and we always denote } $\bm{\Omega}$ {\bfseries be a subgraph of a host graph}  {\bfseries with boundary $\bm{\partial\Omega \not=\emptyset}$.}
 
   The (outward) normal derivative $\frac{\partial u}{\partial_{\Omega} n}(z)$ at $z\in\partial \Omega$ is defined to be 
 \begin{equation}\label{23}
 	\frac{\partial u}{\partial_{\Omega} n}(z):=\sum_{y\in {\Omega}} (u(z)-u(y))\frac{\omega_{zy}}{\mu(z)}.
 \end{equation}

Let $G=(V,E)$ be a finite graph and $\mu:V\to \mathbb{R}$ be a finite measure. The $\omega-$Laplacian $\Delta_{{V}}$ of a function $u:V\to\mathbb{R}$ on a graph $G=(V,E)$ is defined by 
\begin{equation*}
	\Delta_{\omega} u(x)=\Delta_{{V}} u(x):=\sum_{y \in V} (u(y)-u(x))\frac{\omega_{y x}}{\mu(x)},~x\in V,
\end{equation*}
and the $\omega-$Laplacian $\Delta_{\Omega}$ of $u$ on a subgraph $\Omega$ of $G$ is defined by 
\begin{equation}\label{24}
	\Delta_{\omega} u(x)=\Delta_{\Omega} u(x):=\sum_{y \in \bar{\Omega}} (u(y)-u(x))\frac{\omega_{y x}}{\mu(x)},~x\in \Omega.
\end{equation}

In what follows, for an interval $I\subset \mathbb{R}$, we say that a function $f: V\times I\to \mathbb{R}$ belong to $C^{n}(V\times I)$ if for each $x\in V$, the function $f(x,\cdot)$ is $n-$times differentiable in $I$ and $(\frac{d}{dt})^{n} f(x,\cdot)$ is continuous in $I$, and that $f\in L^{1}(V\times I)$, if for each $x\in V$, $f(x,\cdot)$ is integrable in $I$.


Denote a vector $\left( b(x)\right)_{x\in V} $ by $\vec{b}$.
The gradient  $\nabla$ of function $f$ is defined by a vector 
$$\nabla f (x):=\left(  \left[ f(y)-f(x)\right]  \sqrt{\frac{w_{xy}}{2\mu (x)}} \right)_{y\sim x} .$$

 The gradient form of $u$ reads 
\begin{equation*}
	\Gamma(u, v)(x)=\frac{1}{2 \mu(x)} \sum_{y \sim x} \omega_{x y}(u(y)-u(x))(v(y)-v(x)).
\end{equation*}
We denote the length of the gradient of $u$ by
\begin{equation*}
	|\nabla u|(x)=\sqrt{\Gamma(u,u)(x)}=\left(\frac{1}{2 \mu(x)} \sum_{y \sim x} \omega_{x y}(u(y)-u(x))^{2}\right)^{1 / 2}.
\end{equation*}

Denote, for any function $
u: V \rightarrow \mathbb{R}
$, an integral of $u$ on $V$ by $$\int \limits_{V} u d \mu=\sum\limits_{x \in V} \mu(x) u(x).$$ 

Denote  $|V|$ the number of the vertices of the graph $G=(V,E)$ and
$ \text{Vol}(V)=\sum \limits_{x \in V} \mu(x)$ the volume of $V$. For $p \ge 1$, denote $|| u ||_{L^{p}(V)}:=(\int \limits_{V} |u|^{p} d \mu)^{\frac{1}{p}}$. Define a sobolev space and a norm on it by 
\begin{equation*}
	W^{1,2}(V)=\left\{u: V \rightarrow \mathbb{R}: \int \limits_{V} \left(|\nabla u|^{2}+u^{2}\right) d \mu<+\infty\right\},
\end{equation*}
and \begin{equation*}
	\|u\|_{H^{1}(V)}=	\|u\|_{W^{1,2}(V)}=\left(\int \limits_{V}\left(|\nabla u|^{2}+u^{2}\right) d \mu\right)^{1 / 2}.
\end{equation*}

Let $T\in\{\mathbb{R},+\infty\}$, $a\lesssim T$ means that $a<T$ if $T=\infty$; $a\le T$ if $T\le \infty$.  

Let $G=(V,E)$ be a graph, denote $RV$ a set consisting of all the functions on $V$.
\subsection{The eigenvalue problem on networks} 
To solve boundary value problems on graphs, we collect some results about eigenvalues of graph Laplacian operators. For a graph $G=G(V,E)$. We can consider the function $f$ as a $N-$dimensional vector, where $N$ denotes the number of vertices of the graph. Consider in $RV$ the following inner product: for any two functions $f,g\in RV$, set $$(f,g):=\int_{V} fg d\mu.$$  In fact, $-\Delta_{{V}}$ is a positive definite 
 symmetric operator with respect to the inner product. 
By \cite{BH,C,GA}, we obtain the following facts:
\begin{Proposition}\label{20220803} There exist eigenpairs $(\mu_{i},\Psi_{i})$ $(i=0,\cdots,N-1)$ of $-\Delta_{V}$ satisfying the following properties:\\
$(1)$ $0=\mu_0<\mu_{1}\le \mu_{2} \le\cdots \le \mu_{N-1}$;\\
$(2)$ $\sum_{x \in V} \Psi_{i}(x)\Psi_{j}(x)\mu(x)=0,~0 \le i \not= j\le N-1$;\\
$(3)$ $	\sum_{x \in V} |\Psi_{i}(x)|^2 \mu(x)=1,~0\le i \le N-1$.\\
\end{Proposition}

For a subgraph $\Omega$ of a host graph $\bar{\Omega}$ with a weight $\omega$ and $n=|\Omega|$. For the Dirichlet eigenvalue problem \begin{equation}\label{D}
	\begin{cases}
		-\Delta_{\Omega} u=\lambda u~\text{in}~\Omega,\\
		u=0~\text{on}~\Omega
	\end{cases}
\end{equation} and
the Neumann eigenvalue problem 
\begin{equation}\label{343,}
	\begin{cases}
		-\Delta_{\Omega} {u}(x)= K {u}(x),~x\in \Omega,\\
		\frac{\partial {u}}{\partial_{\Omega} n}	(x)  =0,~x\in \partial{\Omega}.
	\end{cases}
\end{equation}
 By \cite{BH,C,GA}, we have the following: 
\begin{Proposition}\label{20220803-1} There exist Dirichlet eigen-pairs $(\lambda_{i},\phi_{i})$ $(i=1,\cdots,n)$ of $-\Delta_{{\Omega}}$ satisfying \eqref{D} and the following properties:\\
	$(1)$ $0<\lambda_1\le\lambda_{2} \le\cdots \le \lambda_{n}$;\\
	$(2)$ $	\sum_{x \in \Omega} {\phi}_{i}(x){\phi}_{j}(x)\mu(x)=0,~1\le i \not= j\le n$;\\
	$(3)$ $	\sum_{x \in \Omega} |{\phi}_{i}(x)|^2 \mu(x)=1,~1\le i \le n$;\\
    $(4)$	$\phi_{1}(x)>0~\text{in}~\Omega.$
\end{Proposition}

\begin{Proposition}\label{Pro1}
There exist Neumann eigen-pairs $(K_{i},\Phi_{i})$ of $-\Delta_{{\Omega}}$ satisfying \eqref{343,} and the following :\\
$(1)$ $0=K_{1}<K_{2}\le \cdots\le K_{n}$;\\
$(2)$ $\sum_{x \in \Omega} {\Phi}_{i}(x){\Phi}_{j}(x)\mu(x)=0,~1\le i \not= j\le n$;\\
$(3)$ $	\sum\limits_{x \in \Omega} |{\Phi}_{i}(x)|^2\mu(x)=1,~1\le i \le n,$.
\end{Proposition}

Hereinafter, we shall assume that $G=G(V,E)$ ia a graph and $\Omega$ is a subgraph of a host graph $G^{'}$.
\section{The maximum principle on finite graphs}

The maximum principles for parabolic and elliptic equations on graphs are going to develop in this section. For this purpose, the existence and uniqueness of linear parabolic and elliptic problems on graphs are discussed in the following subsection.
\subsection{The existence and uniqueness of linear parabolic problems}
We show that now the solvability of the initial-boundary value problem for graphs using the method of separation of variables.

\begin{Theorem}\label{N}
	Let $\Omega$ be a subgraph of a host graph $G$ with $\partial{\Omega}\not= \emptyset$, $g:\partial\Omega\times (0,T)\to \mathbb{R}$ belong to $C^{1}(\partial{\Omega} \times(0,T))\cap L^{1}\left( \partial\Omega\times (0,T)\right) $, $\tilde{G}:\Omega\times (0,T)\to \mathbb{R}$ be a function in $C^{0}(\Omega\times(0,T))\cap L^{1}\left( \Omega\times (0,T)\right)$, $c \in \mathbb{R}$ and $u_0:\Omega \to \mathbb{R}$ be given. Then the following parabolic initial-boundary value problem
	\begin{equation}\label{N1}
		\left\{\begin{array}{l}
			u_{t}(x, t)-\Delta_{\Omega} u(x, t)+cu(x,t)=\tilde{G}(x, t), \quad x \in \Omega, \quad t \in(0, T), \\
			{\mathcal{B}u(x, t)=g(x, t), \quad x \in \partial \Omega, \quad t \in[0, T), }\\
			u(x, 0)=u_0(x), \quad x \in \Omega
		\end{array}\right.
	\end{equation}
	admits a unique solution $u(x,t)\in  C^{1}(\Omega \times(0,T))$. 
\end{Theorem} 
\begin{proof}
	Let $n$ denote the number of vertices of the graph $\Omega$ and $u$ be the solution of \eqref{N1}. In the case of $\mathcal{B}u(x, t)=u$.	Set 
	\begin{equation}\label{1d}
		v(x,t)=\begin{cases}
			u(x,t),~x\in \Omega,~t\in(0,T), \\
			u(x,t)-g(x,t),~x\in \partial \Omega,~t\in(0,T).	
		\end{cases}
	\end{equation}

For $x\in \Omega$, we through straightforward calculations to give  
	\begin{equation*}	
		\begin{aligned}
			\Delta_{{\Omega}} v(x,t)&= \sum\limits_{y\in \overline{\Omega} } \left(  v(y,t)-v(x,t)  \right) \frac{\omega_{yx}}{\mu(x)}  \\
			&=\sum\limits_{y\in \Omega } \left(  u(y,t)-u(x,t)  \right) \frac{\omega_{yx}}{\mu(x)} + \sum\limits_{y\in \partial{\Omega} } \left(  u(y,t)-g(y,t)- u(x,t)  \right) \frac{\omega_{yx}}{\mu(x)},\\
			&=\sum\limits_{y\in \overline{\Omega} } \left(  u(y,t)-u(x,t)  \right) \frac{\omega_{yx}}{\mu(x)} -\sum_{y\in \partial{\Omega}} g(y,t) \frac{\omega_{yx}}{\mu(x)} \\
			&=\Delta_{{\Omega}} u(x,t)-\sum_{y\in \partial{\Omega}} g(y,t) \frac{\omega_{yx}}{\mu(x)}.
		\end{aligned}
	\end{equation*}
	Then $v(x,t)$ satisfies 
	\begin{equation}\label{331}
		\left\{\begin{array}{l}
			v_{t} (x, t)-\Delta_{\Omega} v(x, t)+cv(x,t)=H(x, t), \quad x \in \Omega, \quad t \in(0, T), \\
			v(x, t)=0, \quad x \in \partial \Omega, \quad t \in[0, T), \\
			v(x, 0)=u_0(x), \quad x \in \Omega,
		\end{array}\right.
	\end{equation}
	where $H(x,t):=\tilde{G}(x,t)+\sum\limits_{y\in \partial{\Omega}} g(y,t) \frac{w_{yx}}{\mu(x)}.$
	
	Now, we consider the following expansions
	$$
v(x,t)=\sum\limits_{j=1}^{n} v_{j}(t)\phi_{j}(x),~ 
	H(x,t)=\sum\limits_{j=1}^{n} h_{j}(t)\phi_{j}(x), $$
	
	$$		g(x,t)=\sum\limits_{j=1}^{n} g_{j}(t)\phi_{j}(x),~
	u_0(x)=\sum\limits_{j=1}^{n} u_0^{j} \phi_{j}(x), 
	$$
	where 
	\begin{equation*}
		v_{j}(t)=\sum\limits_{x\in \Omega}v(x,t)\phi_{j}(x)\mu(x),~
		u_0^{j}=\sum\limits_{x\in \Omega} u_0(x)\phi_{j}(x)\mu(x),	
	\end{equation*}
	\begin{equation*}		
		h_{j}(t)=\sum\limits_{x\in \Omega} H(x,t)\phi_{j}(x)\mu(x)
		=  \sum\limits_{x\in \Omega} \left[ \tilde{G}(x,t) + \sum\limits_{y\in \partial{\Omega}} g(y,t) \frac{\omega_{yx}}{\mu(x)}  \right]  \phi_{j}(x)\mu(x)
	\end{equation*}
for $j=1,\cdots,n$.	Then we substitute the above expansions into \eqref{331} to deduce that 
	\begin{equation*}
		\sum\limits_{j=1}^{n} \left[ v_{j}^{\prime}(t)+v_{j} \lambda_{j} +cv_{j}-h_{j}(t) \right]\phi_{j}(x)=0,
	\end{equation*} 
	and
	$
	\sum\limits_{j=1}^{n} [v_{j}(0)-u_0^{j} ]\phi_{j}(x)=0
	$
for $j=1,\cdots,n$. It follows from Corollary \ref{20220803-1} (2) that
	\begin{equation*}\label{}
		\left\{\begin{array}{l}
			v_{j}^{\prime}(t)+v_{j} \lambda_{j} +cv_{j}=h_{j}(t),~t\in (0,T), \\
			v_{j}( 0)=u_0^{j}
		\end{array}\right.
	\end{equation*}
for $j=1,\cdots,n$.	Thus, we deduce that 
	$$
	v_{j}(t)=e^{-(\lambda_{j}+c)t}(u_0^{j}+\int_{0}^{t} e^{(\lambda_{j}+c)s}h_{j}(s) ds  ),~j=1,\cdots,n.
	$$
	Therefore, we see that
	$$
	v(x,t)=\sum_{j=1}^{n} e^{-(\lambda_{j}+c)t}(u_0^{j}+\int_{0}^{t} e^{(\lambda_{j}+c)s}h_{j}(s) ds  ) \phi_{j}(x),~x \in \Omega,~t\in(0,T).
	$$

Due to \eqref{1d}, we conclude that 
	\begin{equation}\label{dj}
		u(x,t)=
		\begin{cases}
			\sum\limits_{j=1}^{n} e^{-(\lambda_{j}+c)t} [u_0^{j}+\int_{0}^{t} e^{(\lambda_{j}+c)s}h_{j}(s) ds  ] \phi_{j}(x),~x \in \Omega,~t\in(0,T), \\
			\sum\limits_{j=1}^{n} e^{-(\lambda_{j}+c)t} [u_0^{j}+\int_{0}^{t} e^{(\lambda_{j}+c)s}h_{j}(s) ds  ] \phi_{j}(x)+g(x,t),~x \in \partial\Omega,~t\in(0,T).
		\end{cases}	
	\end{equation}
A simple calculation shows that the function defined by \eqref{dj} on $\bar{\Omega}\times (0,T)$ gives a solution of equation \eqref{N1}.

In the other case of $\mathcal{B}u=\frac{\partial u}{\partial_{\Omega} n}$. For the host graph $\overline{\Omega}$, let $N$ denote the number of vertices of the graph $\overline{\Omega}$ and $u:\bar{\Omega}\times(0,T)\to \mathbb{R}$ be a solution in $C^{1}(\Omega\times(0,T))\cap C^{0}(\bar{\Omega}\times(0,T))$ of \eqref{N1}. Clearly, there exists a function $I(x,t)\in C^{1}(\Omega \times (0,T))$ such that for any given $t\in(0,T)$,
	$
		\int_{\Omega} I(x,t) d\mu =\int_{\partial \Omega} g(x,t) d\mu $ (If $g\equiv 0$, we take $I\equiv 0$).
 For any fixed $t\in (0,T)$, by a similar argument as in the proof of \cite[Theorem 3.8]{CB}, we see that the following problem 
	\begin{equation*}
		\begin{cases} \Delta_{\Omega} \Phi(x)=I(x,t) & x \in \Omega, \\
			\frac{\partial \Phi}{\partial_{\Omega} n}(x)=g(x,t)  , & x \in \partial{\Omega} \end{cases}
	\end{equation*}
	admits  a solution $\hat{u}(x,t)\in C^{1}(\Omega\times (0,T))$.
	
Set
	\begin{equation*}
			{v(x,t)}:=
			u(x,t)-\hat{u}(x,t)~x\in  \bar{\Omega},~t\in [0,T).\\	
	\end{equation*}  
	By straightforward calculations, we deduce that, for $x\in \Omega$,
	\begin{equation}\label{n}
		\left\{\begin{array}{l}
			v_{t}(x, t)-\Delta_{\Omega} v(x, t)+cv(x,t)=:H(x, t), \quad x \in \Omega, \quad t \in(0, T), \\
			\frac{\partial v}{\partial_{\Omega} n}(x, t)=0, \quad x \in \partial \Omega, \quad t \in[0, T), \\
			v(x, 0)={u}_{0}(x)-\hat{u}(x,0):=\psi(x), \quad x \in \Omega,
		\end{array}\right.
	\end{equation}
	where $H(x,t):=\tilde{G}(x,t)- \hat{u}_{t}(x,t)+\Delta_{{\Omega}} \hat{u}(x,t)-c\hat{u}(x,t) .$

 Recall that $K_j$ and $\Phi_{j}(x)$ $(j=1,2,\cdots,n)$ are the eigenpairs satisfy the eigenvalue problem \eqref{343,}. We consider the following expansions
	\begin{equation*}\label{n1}
		v(x,t)=\sum\limits_{j=1}^{n} v_{j}(t)\Phi_{j}(x), 	
		H(x,t)=\sum\limits_{j=1}^{n} h_{j}(t)\Phi_{j}(x), 
		\psi(x)=\sum\limits_{j=1}^{n} \psi_{j}\Phi_{j}(x). 
	\end{equation*} 
	where 
	$$
		v_{j}(t)=\sum\limits_{x\in \Omega}v(x,t)\Phi_{j}(x)\mu(x),		
			\psi_{j}=\sum\limits_{x\in \Omega} \psi(x)\Phi_{j}(x)\mu(x),
$$
$$	h_{j}(t)=\sum\limits_{x\in \Omega} H(x,t)\Phi_{j}(x)\mu(x)
			=  \sum\limits_{x\in \Omega} \left[ G (x,t)  - \hat{u}_{t}(x,t)+\Delta_{{\Omega}} \hat{u}(x,t)-c\hat{u}(x,t)  \right]  \Phi_{j}(x)\mu(x).
	$$
	Substituting the above expansions into \eqref{n}, we deduce that 
$$
		\sum\limits_{j=1}^{n} [v_{j}^{\prime}(t)+v_{j} K_{j} +cv_{j}-h_{j}(t)]\Phi_{j}(x)=0,
\sum\limits_{j=1}^{n} [v_{j}(0)-\psi_{j} ]\Phi_{j}(x)=0.
$$

In view of the properties of $\{ \Phi_{j}\}_{j=1}^{n}$ in Proposition \ref{Pro1}, we deduce that 
$$
		v_{j}^{\prime}(t)+v_{j} K_{j} +cv_{j}-h_{j}(t)=0,~t\in(0,T),
		v_{j}(0)-\psi_{j}=0,~j=1,\dots,n.
$$
	Thus, we see that 
$$
		v_{j}(t)=e^{-(K_{j}+c)t}[\psi_{j}+\int_{0}^{t} e^{(K_{j}+c)s}h_{j}(s) ds  ],~j=1,\cdots,n.
$$
	Therefore, we conclude that
$$
		v(x,t)=\sum_{j=1}^{n} e^{-(K_{j}+c)t} [\psi_{j}+\int_{0}^{t} e^{(K_{j}+c)s}h_{j}(s) ds ]  \Phi_{j}(x),~x \in \Omega,~t\in [0,T).
$$
	It follows that 
	\begin{equation*}\label{dj22}
		u(x,t)=
			\sum\limits_{j=1}^{n} e^{-(K_{j}+c)t}[\psi_{j}+\int_{0}^{t} e^{(K_{j}+c)s}h_{j}(s) ds  ] \Phi_{j}(x)+\hat{u}(x,t),~ x \in  \bar{\Omega},~t\in[0,T). 
	\end{equation*}
	A simple calculation shows that the function $u$ defined by \eqref{dj22} on $\bar{\Omega}\times [0,T)$ gives a solution of equation \eqref{N1}.
\end{proof}

\begin{Remark}
	In Theorem \ref{N}, if we replace the assumptions that $g \in C^{1}(\partial{\Omega} \times(0,T)) \cap L^{1}\left(\partial \Omega\times (0,T)\right)$ and $\tilde{G}\in C^{0}(\Omega\times(0,T))\cap L^{1}\left( \Omega\times (0,T)\right) $ with $g \in C^{1}(\partial{\Omega} \times[0,T)) \cap L^{1}\left(\partial \Omega\times (0,T)\right)$ and $\tilde{G}\in C^{0}(\Omega\times[0,T))\cap L^{1}\left( \Omega\times (0,T)\right) $ respectively, then the regularity of the unique solution $u(x,t)$ can be increased to $u(x,t)\in C^{1}(\Omega\times [0,T) )$. 
\end{Remark}

Next, we establish the existence and uniqueness for the Cauchy problem on graphs.
\begin{Theorem}\label{3.3}
	Let $G=(V,E)$ be a graph, $M \in \mathbb{R}$, $h\in C^{0}\left( V\times [0,T)\right)\cap L^{1}\left( V\times (0,T)\right)  $ and $\psi:V \to \mathbb{R}$ be given. Then the following Cauchy problem
	\begin{equation}\label{56}
		\left\{\begin{array}{l}
			u_{t}(x, t)-\Delta_{V} u(x, t)+Mu(x,t)={h}(x, t), \quad x \in V, \quad t \in(0, T), \\
			u(x, 0)=\psi(x), \quad x \in V
		\end{array}\right.
	\end{equation}
	admits a unique solution $u(x,t)\in C^{0}(V \times[0,T))$. 
\end{Theorem} 
\begin{proof}
	Let $N$ denote the number of vertices of $V$ and  $u(x,t)$ be a solution to \eqref{56}. 
We consider the following expansions 
	\begin{equation}\label{59}
		u(x,t)=\sum_{j=0}^{N-1} u_{j}(t) \Psi_{j}(x),~	
		h(x,t)=\sum_{j=0}^{N-1} h_{j}(t)\Psi_{j}(x),~	
		\psi(x)=\sum_{j=0}^{N-1} \psi_{j}(t) \Psi_{j}(x),
	\end{equation}
	where \begin{equation*}
		u_{j}(t)=\sum\limits_{x\in V} u(x,t)\Psi_{j}(x)\mu(x),~
		h_{j}(t)=\sum\limits_{x\in V} h(x,t)\Psi_{j}(x)\mu(x),~
		\psi_{j}=\sum\limits_{x\in V} \psi(x)\Psi_{j}(x)\mu(x).
	\end{equation*}  

Substituting \eqref{59} into \eqref{56}, we conclude that 
$$
		\sum\limits_{j=0}^{N-1} \left[ u_{j}^{\prime}(t)+u_{j} \mu_{j} +Mu_{j}-h_{j}(t) \right]\Psi_{j}(x)=0,
		\sum\limits_{j=0}^{N-1} [u_{j}(0)-\psi_{j} ]\Psi_{j}(x)=0.
$$
By Proposition \ref{20220803} (2), we know that 
	\begin{equation*}\label{}
		\left\{\begin{array}{l}
			u_{j}^{\prime}(t)+u_{j} \mu_{j} +Mu_{j}=h_{j}(t),~t\in (0,T), \\
			u_{j}( 0)=\psi_{j}
		\end{array}\right.
	\end{equation*}
	$\text{for}~j=0,1,\cdots,N-1.$ Thus, we obtain 
$$
		u_{j}(t)=e^{-(\mu_{j}+M)t}(\psi_{j}+\int_{0}^{t} e^{(\mu_{j}+M)s}h_{j}(s) ds  ),~j=0,1,\cdots,N-1.
$$
	Therefore, 
	\begin{equation}\label{cj}
		u(x,t)=\sum_{j=0}^{N-1} e^{-(\mu_{j}+M)t}\left[ \psi_{j}+\int_{0}^{t} e^{(\mu_{j}+M)s}h_{j}(s) ds  \right]  \Psi_{j}(x).
	\end{equation}
Furthermore, it is easy to see that $u(x,t)$ in \eqref{cj} satisfies the Cauchy problem \eqref{56}.
\end{proof}

\subsection{Maximum principle for parabolic problems}
This section develops the maximum principle for elliptic equations and parabolic equations on finite connected weighted graphs.

\subsubsection{The parabolic equation with initial-boundary value}

This subsection is devoted to establish the maximum principle for parabolic equations with boundary conditions on graphs.
\begin{Theorem}\label{max} {\rm (Maximum principle)}
	Let $\Omega$ be a subgraph of a host graph $G=(V,E)$ with $\partial\Omega \not= \emptyset$. For any $T>0$, we assume that $v(x,t)$ $\in C^{1}(\Omega\times(0,T] ) \cap C^{0}(\Omega\times [0,T])$ satisfies 
	\begin{equation*}
		\begin{cases} v_{t}(x,t )-\Delta_{\Omega} v(x,t)-k(x,t ) v(x,t ) \geq 0, & (x,t) \in \Omega \times (0, T] , \\ v(x,0) \geq 0, & x \in \Omega, \\ \mathcal{A}v( x,t) \geq 0, & ( x,t) \in  \partial \Omega \times[0, T], \end{cases}
	\end{equation*}
where $k(x,t)$ is bounded on $\Omega \times (0,T]$ and $\mathcal{A}v=\alpha(x,t) v+\beta(x,t)\frac{\partial v}{\partial_{\Omega} n}$, $\alpha\ge 0 $, $\beta\ge 0$, and $\alpha+\beta>0$ on $\Omega\times [0,T]$. Then
\begin{equation}\label{1,}
	v(x,t)\ge 0~ \text{on}~ \bar{\Omega}\times[0,T] .
\end{equation} 
\end{Theorem}
\begin{proof}	
	Set $$\mathcal{L}v(x,t)=v_{t}(x,t)-\Delta_{\Omega} v(x,t)-k(x,t)v(x,t).$$ It's clear that $\mathcal{L} v\ge 0~\text{in}\;\Omega\times (0,T].$ As $k(x,t)$ is bounded in $\Omega\times (0,T]$, we can find a positive number $l$ such that $l>k(x,t)$ for all $(x,t)\in \Omega\times (0,T]$. 
	
Define $\psi=ve^{-lt}$. It gives that 
	 \begin{equation*}
	 	\begin{aligned}
	 		\mathcal{L} v(x,t)&=v_{t}(x,t)-\Delta_{\Omega} v(x,t)-k(x,t)v(x,t) \\
	 		&=e^{lt}(\psi_{t}-\Delta_{\Omega} \psi-(k-l)\psi)	\\
	 		&=:e^{lt}\tilde{\mathcal{L}} \psi(x,t).
	 	\end{aligned} 
	 \end{equation*}
Recall that $v(x,t)$ belongs to $C^{0}(\Omega \times[0,T])$, then we have ${\psi(x,t)\in C^{0}(\Omega \times [0,T])}$. Therefore, we can find $(x_0 , t_0)\in \Omega\times[0,T]$ so that 
 \begin{equation}\label{m}
 \psi(x_0 , t_0)=\min\limits_{\Omega\times [0,T]} \psi.
 \end{equation} 
 
 Next, we assert that $\psi(x_0 , t_0)\ge 0$. Suppose by way of contradiction that $\psi(x_0 , t_0)<0$. 
 If $t_0=0$, it's clear that $\psi(x_0 , t_0)=v(x_0,0) < 0$. This is a contradiction.
 For $t_0 \in (0,T]$, then we use $k(x_0 , t_0)<l$ and $\psi(x_0 , t_0)<0$ to obtain that 
 \begin{equation*}
 	\tilde{\mathcal{L}} \psi(x_0 , t_0)<\psi_{t}(x_0 , t_0)-\Delta_{\Omega} \psi (x_0 , t_0).
 \end{equation*} 
It's obviously that $\psi_{t}(x_0 , t_0)\le 0$. Then, we have 
\begin{equation}\label{3}
	\tilde{\mathcal{L}} \psi (x_0 , t_0) < -\Delta_{\Omega} \psi (x_0 , t_0).
\end{equation}

For any $y\in \partial \Omega$, we use
\begin{equation*}
	\left( \alpha \psi+\beta\frac{\partial \psi}{\partial_{\Omega} n} \right) \bigg| _{(y,t_0)} =e^{-lt_0}\left( \alpha v+\beta\frac{\partial v}{\partial_{\Omega} n} \right) \bigg| _{(y,t_0)}  \ge 0,
\end{equation*}
to conclude that $\psi(y,t_0)\ge \psi (x_0, t_0)$ for all $y\in \partial {\Omega}$. This implies that 
\begin{equation}\label{4}
	\Delta_{\Omega} \psi\left(x_{0}, t_{0} \right)=\frac{1}{\mu\left(x_{0}\right)} \sum_{y:y \sim x_{0}} \omega_{x_{0} y}\left(\psi\left(y,t_{0} \right)-\psi\left( x_{0},t_{0}\right)\right) \geq 0.
\end{equation}
Note that $\mathcal{L}v\ge 0$ in $\Omega\times (0,T]$, we see that $\tilde{\mathcal{L}} \psi \ge 0$ in $\Omega\times (0,T]$.
On the other hand, it follows from \eqref{3} and \eqref{4} that 
$$\tilde{\mathcal{L}} \psi(x_0 , t_0) <0. $$ This is impossible.
Thus, we can conclude that $\psi(x_0,t_0)\ge 0$ and hence that $\psi(x,t)\ge 0$ for $x\in \Omega$ and $t\in(0,T]$ by \eqref{m}. 

If there exists $(x_1,t_1) \in \partial \Omega \times [0,T] $ such that $\psi(x_1,t_1)<0$,
then we have 
\begin{equation*}
\left( 	\alpha\psi+\beta\frac{\partial \psi}{\partial_{\Omega} n} \right)  (x_1,t_1)=\left( \alpha\psi\right) (x_1,t_1)+\beta(x_1, t_1) \frac{1}{\mu(x_1)} \sum_{z \in \Omega}(\psi (x_1,t_1)-\psi (z,t_1)) \omega_{x_1 z} < 0,
\end{equation*}
which is a contradiction. Thus, we obtain ${u(x,t)\ge 0}$ for $x\in \bar{\Omega}$ and $t\in [0,T]$.
\end{proof}

\begin{Theorem}\label{sm}
	{\rm (Strong maximum principle)} Suppose that $u\in C^{1}({\Omega}\times (0,T])\cap C^{0}(\bar{\Omega}\times [0,T])$ satisfies 
	\begin{equation*}
		\begin{cases}u_{t}-\Delta_{\Omega} u+k(x,t)u\ge 0, & (x, t) \in \Omega \times(0, T], \\ 
		\mathcal{B}u(x, t)\ge0, & (x, t) \in \partial \Omega \times(0, T], \\
		 u(x, 0)=u_{0}(x) \geqslant, \not\equiv 0, & x \in \Omega. \end{cases}
	\end{equation*} 
	and $k(x,t)$ is bounded function for $x\in \Omega$ and $t\in [0,T]$, where $\mathcal{B}u=u$ or $\mathcal{B}u=\frac{\partial u}{\partial_{\Omega} n}$. Then if $\mathcal{B}u=u$, then $u(x,t)> 0$ for $x\in {\Omega}$ and $t\in (0,T]$ and if $\mathcal{B}u=\frac{\partial u}{\partial_{\Omega} n}$, then $u(x,t)> 0$ for $x\in \bar{\Omega}$ and $t\in (0,T]$.
\end{Theorem} 
 \begin{proof}
Using Theorem \ref{max}, we obtain $u(x,t)\ge 0$ on $\bar{\Omega} \times [0,T]$. As the boundness of $k(x,t)$, we can find a constant $\bar{k}>0$ such that $k\le\bar{k}$. It follows that 
 $$
 			0\le u_{t}-\Delta_{\Omega} u +ku 
 			\le u_{t}-\Delta_{\Omega} u+\bar{k}u 
$$
in $\Omega \times [0,T]$.
By Theorems \ref{N} and \ref{max}, the problem
 		\begin{equation*}
 		\begin{cases}v_{t}-\Delta_{\Omega} v+\bar{k}v= 0, & (x, t) \in \Omega \times(0, T], \\ \mathcal{B}v(x, t)=0, & (x, t) \in \partial \Omega \times[0, T], \\ v(x, 0)=u_{0}(x), & x \in \Omega \end{cases}
 	\end{equation*}
 admits a unique solution $v\ge 0$ on $\bar{\Omega}\times (0,T]$. It follows that 
 \begin{equation*}
 	\begin{aligned}
 		 v_{t}-\Delta_{\Omega} v +kv
 		&\le v_{t}-\Delta_{\Omega} v+\bar{k}v=0 ~\text{in}~\Omega \times [0,T].
 	\end{aligned}
 \end{equation*}
 
Setting $W=u-v$, we find that $W$ satisfies
\begin{equation*}
	\begin{cases}W_{t}-\Delta_{\Omega} W+\bar{k} W\ge 0, & (x, t) \in \Omega \times(0, T], \\ \mathcal{B}W(x, t)=0, & (x, t) \in \partial \Omega \times[0, T], \\ W(x, 0)=0, & x \in \Omega. \end{cases}
\end{equation*}
Then we apply Theorem \ref{max} to conclude that $W\ge 0$ on $\bar{\Omega}\times [0,T]$. It's easily seen that $u\ge v$ on $\bar{\Omega}\times [0,T]$.

Consider the transformation $w=e^{\bar{k} t} v$,  then $w$ satisfies
 \begin{equation*}
 	\begin{cases}w_{t}-\Delta_{\Omega} w= 0, & (x, t) \in \Omega \times(0, T], \\ \mathcal{B}w(x, t)=0, & (x, t) \in \partial \Omega \times[0, T], \\ w(x, 0)=u_{0}(x) \geqslant, \not\equiv 0, & x \in \Omega. \end{cases}
 \end{equation*}
Denote $n=|\Omega|$. For the Dirichlet boundary $\mathcal{B}u=u$,
it follows from Theorem \ref{N} that
\begin{equation}\label{452,}
	w(x,t)=\sum\limits_{y\in \Omega}P(x,y,t)u_{0}(y)~\text{for}~x\in\bar{\Omega}~\text{and}~t\in[0,T],
\end{equation}
where  $P(x,y,t)=\sum\limits_{j=1}^{n}e^{-\lambda_{j}t}\phi_{j}(x)\phi_{j}(y)\mu(y)$. It is easy to check that $$P_{t}=\Delta_{\Omega,x}P,$$ where 
$
	\Delta_{\Omega,x} P(x,y,t)=\sum_{z \in \bar{\Omega}}[P(z,y,t)-P(x,y,t)]\frac{\omega_{z x}}{\mu(x)}.
$
Combining Theorem \ref{max} and 
 $$P(x,y,0)=\delta_{x}(y)=\begin{cases}
	1,x=y,\\0,x\not = y
\end{cases}\text{for}\;x,y\in \Omega,$$
we have  $P(x,y,t)\ge0$ for $t\in[0,T]$ and $x,y\in \Omega$. 

Now,we shall show 
\begin{equation}\label{22,}
	P(x,y,t)>0~\text{ for} ~x,y\in \Omega~\text{ and}~ 0<t\le T
\end{equation}
by using proof by contradiction. Otherwise, there exists $x_0,y_0\in \Omega$ and $t_0\in(0,T]$ such that $P(x_0,y_0,t_0)=0$. It's easily seen that  
$$0=P_{t}(x_0 ,y_0,t_0)=\Delta_{\Omega,x} P(x_0,y_0,t_0)\ge 0.$$ 
and hence $P(x,y_0,t_0)=P(x_0,y_0,t_0)=0$ $\forall x\in \Omega:x\sim x_0$. Note that
$\Omega$ is connected, for any $z\in \Omega$, we can find $x_j (j=1,\cdots,n)$ such that $x_1\sim x_2\sim\cdots\sim x_n\sim z$. It now follows from the discussion inductively that $P(z,y_0,t_0)=0$, which implies that 
$$P(y_0,y_0,t_0)=\sum\limits_{j=1}^{n}e^{-\lambda_{j}t_0}\phi_{j}^{2}(y_0)\mu(y_0)=0.$$ Therefore, we see that 
$\sum\limits_{j=1}^{n}\phi_{j}^{2}(y_0)\mu(y_0)=0$ which is a contradiction with Corollary \ref{20220803-1} (4).
Combining \eqref{452,}, \eqref{22,} and $u_{0}\ge,\not\equiv 0$, we obtain
$w(x,t)>0$ for $x\in \Omega$ and $t\in (0,T]$. It follows that $v(x,t)>0$ for $x\in \Omega$ and $t\in (0,T]$. Recalling that $u\ge v$ on $\Omega\times [0,T]$, we see that $u(x,t)>0$ for $x\in \Omega$ and $t\in (0,T]$.

Next, we remain to prove the conclusion holds with the Neumman boundary $\mathcal{B}u=\frac{\partial u}{\partial_{\Omega} n}.$ 
Applying Theorem \ref{N}, we see that 
\begin{equation*}\label{452,,}
	w(x,t)=\sum\limits_{y\in \Omega}Q(t,x,y)u_{0}(y),
\end{equation*}
where  $Q(t,x,y)=\sum\limits_{j=1}^{n}e^{-K_{j}t}\Phi_{j}(x)\Phi_{j}(y)\mu(y)$. It is easy to check that $$Q_{t}=\Delta_{\Omega,x}Q,$$ where 
$
\Delta_{\Omega,x} Q(t,x,y)=\sum_{z \in \bar{\Omega}}[Q(t,z,y)-Q(t,x,y)]\frac{\omega_{z x}}{\mu(x)}.
$
Due to $$Q(0,x,y)=\delta_{x}(y)=\begin{cases}
	1,x=y,\\0,x\not = y,
\end{cases}\text {for}\;x,y\in \Omega,$$
we use the similar approach to that in the case that $\mathcal{B}u=u$ to deduce that $Q(t,x,y)>0$ for $x,y\in \Omega$ and $0<t\le T$. It follows that $w(x,t)>0$ for $x\in \Omega$ and $t\in (0,T]$ which implies that 
\begin{equation}\label{z1}
	u(x,t)>0~\text{ on} ~\Omega \times (0,T].
\end{equation}

If there exists $x_0 \in \partial{\Omega}$ and $t_0 \in (0,T]$ such that $u(x_0, t_0)=0$, then 
\begin{equation*}
	0=\frac{\partial u}{\partial_{\Omega} n}(x_0, t_0)=\sum_{y \in \bar{\Omega}} (u(x_0 , t_0)-u(y, t_0))\frac{\omega_{x_{0} y}}{\mu(x_0)} \le 0.
\end{equation*}
Therefore, we have 
\begin{equation*}
	u(y,t_0)=u(x_0, t_0)=0~\forall y\in {\Omega}:~y\sim x_0,
\end{equation*}
which contradicts to \eqref{z1}. Thus, we obtain
$$u(x,t)>0~\text{for}~x\in\bar{\Omega}~\text{and}~t\in(0,T].$$
 \end{proof}

\subsubsection{The parabolic equation with initial value}  This subsection is devoted to establish the maximum principle for parabolic equations under the condition that there is no boundary condition on graphs.

\begin{Theorem}\label{maxc} {\rm (Maximum principle)}
	Let $G=(V,E)$ be a graph. For any $T>0$, we assume that $v(x,t)$ $ \in C^{1}(V \times [0,T]) $, which satisfies 
	\begin{equation*}
		\begin{cases}\frac{\partial}{\partial t} v(x,t )-{\Delta_{V} }v(x,t)-k(x,t ) v(x,t ) \geq 0, & (x,t) \in V \times (0, T] , \\ v(x,0) \geq 0, & x \in V, \end{cases}
	\end{equation*}
	where $k(x,t)$ is bounded on $V \times [0,T]$. Then
	\begin{equation*}\label{1,c}
		v(x,t)\ge 0~ \text{on}~ V \times[0,T] .
	\end{equation*} 
\end{Theorem}
\begin{proof}
Denote $$\mathcal{L}v(x,t)=v_{t}(x,t)-\Delta_{\Omega} v(x,t)-k(x,t)v(x,t).$$ It's easily seen that $\mathcal{L} v\ge 0~\text{in}~V\times (0,T].$ We observe that $k(x,t)$ is bounded in $V\times (0,T]$, then there exists a positive constant $l$ such that $l>k(x,t)$ for all $(x,t)\in V\times (0,T]$. 

If we set $\psi=ve^{-lt}$, then 
	\begin{equation*}
		\begin{aligned}
			\mathcal{L} v(x,t)&=v_{t}(x,t)-\Delta_{\Omega} v(x,t)-k(x,t)v(x,t) \\
			&=e^{lt}(\psi_{t}-\Delta_{\Omega} \psi-(k-l)\psi)	\\
			&=:e^{lt}\tilde{\mathcal{L}} \psi(x,t).
		\end{aligned} 
	\end{equation*}
We use $\mathcal{L}v\ge 0$ in $V\times (0,T]$ easily to deduce that $\tilde{\mathcal{L}} \psi \ge 0$ in  $V \times(0,T]$.
Recalling that {$v(x,t)\in C^{1}(V \times[0,T])$}, we have {$\psi(x,t)\in C^{1}(V \times[0,T])$}. Then we can find $(x_0 , t_0)\in V\times[0,T]$ such that 
	\begin{equation}\label{m1}
		\psi(x_0 , t_0)=\min\limits_{V \times[0,T]} \psi.
	\end{equation} 

Now, we are going to prove that $\psi(x_0 , t_0)\ge 0$. Arguing indirectly, we assume that  $\psi(x_0 , t_0)<0$. 
If $t_0=0$, we have $\psi(x_0 , t_0)=v(x_0,0) \ge 0$. This is a contradiction.
For $t_0 \in (0,T]$, we use $k(x_0 , t_0)<l$ and $\psi(x_0 , t_0)<0$ to deduce that 
	\begin{equation*}\label{le}
		\tilde{\mathcal{L}} \psi(x_0 , t_0)<\psi_{t}(x_0 , t_0)-\Delta_{\Omega} \psi (x_0 , t_0).
	\end{equation*} 
In this situation, we observe that $\psi_{t}(x_0 , t_0)\le 0$. Then, we have 
	\begin{equation}\label{21}
		\tilde{\mathcal{L}} \psi (x_0 , t_0) < -\Delta_{\Omega} \psi (x_0 , t_0).
	\end{equation}
It follows from \eqref{m1} that
	\begin{equation*}
		-\Delta_{{\Omega}} \psi (x_0, t_0) = -\sum_{y:y \sim x_{0}} (\psi(y,t_0) -\psi(x_0 , t_0)) \frac{\omega_{x_{0} y}}{\mu(x_0)}\le 0.
	\end{equation*}

Moreover, we apply  \eqref{21}
to deduce that 
	\begin{equation*}
		\tilde{\mathcal{L}} \psi(x_0, t_0)<0.
	\end{equation*}
This contradicts with  $\tilde{\mathcal{L}} \psi \ge 0$ in $V \times(0,T]$.
Therefore, we conclude that $\psi(x_0,t_0)\ge 0$ and hence that $\psi(x,t)\ge 0$ on $V \times[0,T]$. This implies that $v\ge 0$ on $V \times[0,T]$.	
\end{proof}

\begin{Theorem}\label{smc}
	{\rm (Strong maximum principle)}
	Assume that $u(x,t)\in C^{1}(V\times[0,T])$ satisfies 
	\begin{equation*}
		\begin{cases}
			u_{t}-\Delta_{{V}} u+ k(x,t)u\ge 0,~(x,t)\in V\times(0,T],\\
			u(x,0)=u_0(x)\ge,\not\equiv 0,~x\in V,
		\end{cases}
	\end{equation*}
where $k$ is bounded on $V\times [0,T]$.
Then we have $u(x,t)>0$ for $x\in V$ and $t\in (0,T]$.
\end{Theorem}
\begin{proof}
	It follows from Theorem \ref{maxc} that $u(x,t)\ge 0$ on $V\times[0,T]$. Thanks to $k(x,t)$ is bounded, we can find $\bar{k}>0$ such that $k\le \bar{k}$. Using Theorems \ref{3.3} and \ref{maxc}, the linear parabolic problem 
	\begin{equation*}
		\begin{cases}
			v_{t}-\Delta_{{V}} v+ \bar{k}v= 0,~(x,t)\in V\times[0,T],\\
			v(x,0)=u_0(x),~x\in V.
		\end{cases}
	\end{equation*}
admits a unique solution $v\ge 0$ on $V\times[0,T]$.
Obviously, 
$$			v_{t}-\Delta_{V} v +kv
			\le v_{t}-\Delta_{V} v+\bar{k}v=0 
$$
in $V \times (0,T]$.

Denote $W=u-v$. Then we find that $W$ satisfies
	\begin{equation*}
		\begin{cases}W_{t}-\Delta_{V} W+k W\ge 0, & (x, t) \in {V}\times(0, T], \\  W(x, 0)=0, & x \in {V}. \end{cases}
	\end{equation*}
	By Theorem \ref{maxc}, we deduce that $w\ge 0$ on $V \times [0,T]$, and hence  that 
	 \begin{equation}\label{b1,}
	u(x,t)\ge v(x,t)~ \text{on} ~V\times[0,T].
\end{equation}

Defining $w=ve^{\bar{k}t}$; then we know that $w$ is the solution of
\begin{equation*}
	\begin{cases}
	w_t -\Delta_{V}w=0,~x\in V,~t\in [0,T],\\	w(x,0)=u_0(x),~x\in V.
	\end{cases}	
\end{equation*}
and
\begin{equation}\label{452}
	w(x,t)=\sum\limits_{y\in V}p(x,y,t)u_{0}(y)
\end{equation}
by Theorem \ref{3.3}, where  $p(x,y,t)=\sum\limits_{j=0}^{N-1}e^{-\mu_{j}t}\Psi_{j}(x)\Psi_{j}(y)\mu(y)$ is called the heat kernel. It is easy to check that 
$$p_{t}=\Delta_{\Omega,x}p,$$ where 
$
	\Delta_{\Omega,x} p(x,y,t)=\sum_{z \in \bar{\Omega}}[p(z,y,t)-p(x,y,t)]\frac{\omega_{z x}}{\mu(x)}.
$
It's known that 
$$p(x,y,0)=\delta_{x}(y)=\begin{cases}
	1,x=y,\\0,x\not = y
\end{cases}\text{for}\; x,y\in V
$$
Moreover, we apply Theorem \ref{maxc} to obtain  $p(x,y,t)\ge0$ for $t\in[0,T]$ and $x,y\in V$. 

Now, we shall claim that 
\begin{equation}\label{22}
	p(x,y,t)>0~\text{ for} ~x,y\in V~\text{ and}~ 0<t\le T.
\end{equation}
 Otherwise, there exists $x_0,y_0\in \Omega$ and $t_0\in(0,T]$ such that $p(x_0,y_0,t_0)=0$. Thus we deduce that  
 $$0=p_{t}(x_0 ,y_0,t_0)=\Delta_{\Omega,x} p(x_0,y_0,t_0)\ge 0.$$ 
 Therefore, $p(x,y_0,t_0)=p(x_0,y_0,t_0)=0$ $\forall x\in V:x\sim x_0$. As the connectivity of
 $V$, for any $z\in V$, we can find $x_j (j=1,\cdots,n)$ such that $x_1\sim x_2\sim\cdots\sim x_n\sim z$. By the discussion inductively, we see that $p(z,y_0,t_0)=0$. This implies that 

$$p(y_0,y_0,t_0)=\sum\limits_{j=0}^{N-1}e^{-\mu_{j}t_0}\Psi_{j}^{2}(y_0)\mu(y_0)=0.$$ Therefore, we see that 
$\sum\limits_{j=0}^{N-1}\Psi_{j}^{2}(y_0)\mu(y_0)=0.$ This is a contradiction with Proposition \ref{20220803}.

Thanks to $u_{0}\ge,\not\equiv 0$, by \eqref{452} and \eqref{22}, we deduce that
$w(x,t)>0$ for $x\in V$ and $t\in (0,T]$. It follows that $v(x,t)>0$ for $x\in V$ and $t\in (0,T]$. In view of \eqref{b1,}, we see that $u(x,t)>0$ for $x\in V$ and $t\in (0,T]$.
\end{proof}

\subsection{Maximum principle for elliptic problems}
In this section, we establish the maximum principle for elliptic equations on graphs.

\subsubsection{The elliptic problems with boundary}
\begin{Theorem}{\rm (Maximum principle)}\label{25}
	Suppose that $u$ satisfies 
	\begin{equation*}
		-{\Delta_{\Omega} u}- \vec{b}\cdot\nabla u +c(x) u\ge 0\text{~in~}\Omega,~\mathcal{B} u\ge 0~\text{on}~\partial{\Omega}.
	\end{equation*}
where $b(x)\ge 0$ and $c(x)> 0$ for all $x\in V$.
Then we have $u\ge 0$ for all $x\in\bar{\Omega}$.
\end{Theorem}
\begin{proof}
	Let $u(x_0)=\min\limits_{\bar{x\in\Omega}} u(x)$. We shall prove that $u(x_0)\ge 0$. Suppose by way of contradiction that $u(x_0)<0$. 
	If $x_0\in \Omega$, then we see that 
	\begin{equation*}
		\begin{aligned}
			&-\Delta_{{\Omega}} u(x_0)-\vec{b}\cdot\nabla u(x_0)\\
			&=-\sum\limits_{y:y\sim x_0} (u(y)-u(x_0)) \frac{\omega_{y x_0}}{\mu(x_0)}-\sum_{y \in \bar{\Omega}} b(y)\left(u(y)-u(x_0) \right)\sqrt{\frac{\omega_{x_{0} y}}{2\mu(x_0)}}
			 \le 0.
		\end{aligned}	
	\end{equation*}
Thus, we have 
\begin{equation*}
	\left( -{\Delta_{\Omega}} u-\vec{b}\cdot \nabla u +c(x)u\right)  \big{|}_{x=x_0} <0,
\end{equation*}
which is a contradiction. 

For $x_0\in \partial{\Omega}$, we first consider the case of $\mathcal{B} u=u$, we deduce that $u(x_0)<0$, which contradicts to $u\ge 0$ on $\partial{\Omega}$. 
However, for the case of $\mathcal{B}u=\frac{\partial u}{\partial_{\Omega} n}$, we see that 
\begin{equation*}
	0\le \frac{\partial u}{\partial_{\Omega} n} (x_0)= \sum\limits_{y\in \bar{\Omega}:y\sim x_0}  (u(x_0)-u(y))\frac{\omega_{x_0 y}}{\mu (x_0)} \le 0.
\end{equation*}
This implies that $u(y)=u(x_0)$ for all $y\in \Omega$ satisfying $y\sim x_0$. As $x_0\in \partial \Omega$, we can find  $z\in \Omega:z\sim x_0$ such that $u(z)=u(x_0)<0$. By a similar argument as above (the case that $x_0\in \Omega$), we obtain a contradiction. 
Therefore, we deduce that $u(x_0)\ge 0$ and hence that $u(x)\ge 0$ $\forall x \in \bar{\Omega}$.
\end{proof}

\subsubsection{The elliptic problems without boundary}
\begin{Theorem}\label{27}
	Assume that $u$ satisfies $-\Delta_{V} u- \vec{b}\cdot\nabla u +c(x) u\ge 0$ on $V$, where $b(x)\ge 0$ and $c(x)> 0$ for all $x\in V$. Then $u\ge 0$ for all $x\in V$.
\end{Theorem}
\begin{proof}
	Assume that there exists $x_0\in V$ such that $u(x_0)=\min\limits_{V} u<0$. Then we see that
	\begin{equation*}
		-\Delta_{V} u(x_0)-\vec{b}\cdot \nabla u(x_0)+c(x_0)u(x_0)<0,
	\end{equation*}
	which is a contradiction.
\end{proof}

\section{Upper and lower solutions method}

In this section, we develop the upper and lower solutions method for parabolic equations and elliptic equations on finite connected weighted graphs.

\subsection{The upper and lower solutions method for parabolic problems}
\subsubsection{The parabolic equation with initial-boundary value}

In this subsection, we introduce the upper and lower solutions method for the problem \eqref{1s} on graphs.
 
 \begin{Definition}
 	Suppose that $u\in C^{1}(\Omega\times[0,T]) \cap C^{0}(\bar{\Omega}\times[0,T]) $. If $u$ satisfies 
 	\begin{equation*}
 		\begin{cases}
 				u_t-\Delta_{\Omega} u\ge (\le) f(x,t,u),~&(x,t)\in \Omega\times (0,T], \\
 		\mathcal{B}u(x,t) \ge (\le) g,~&(x,t)\in \partial \Omega\times(0,T] ,\\
 		u(x,0)\ge (\le) \psi,~&x\in \Omega,		
 		\end{cases}	
 	\end{equation*}
 then we call $u$ an upper (a lower) solution of the problem \eqref{1s}.
\end{Definition}
 
 \begin{Theorem}{\rm (Ordering of upper and lower solutions)} \label{o}
 	Let $\bar{u}$, $\underline{u}$ be an upper and a lower solution of \eqref{1s}, respectively. Denote $\sigma=\inf\limits_{\Omega_T}\min\{\bar{u},\underline{u} \}$ and $\omega=\sup\limits_{\Omega_T} \max\{\bar{u},\underline{u} \}$. If $f$ satisfies the Lipschitz condition in $u\in [\sigma,\omega]$, i.e., there exists a constant $M>0$ such that \begin{equation}\label{4.5}
 		|f(x,t,u)-f(x,t,v)| \le M |u-v|,~\forall~(x,t)\in \Omega\times[0,T],~u,v\in [\sigma,\omega], 
 	\end{equation}then $\bar{u}\ge \underline{u}$ in $\bar{\Omega}\times [0,T]$. Moreover, $\bar{u}>\underline{u}$ in $\Omega \times (0,T]$ if $\bar{u}(x,0)\not\equiv \underline{u}(x,0)$.
 \end{Theorem}
\begin{proof}
	Setting 
	\begin{equation*}
		\hat{c}(x,t)=\begin{cases}
			M,~\text{if}~\bar{u}(x,t)\ge \underline{u}(x,t), \\
			-M,~\text{if}~\bar{u}(x,t)< \underline{u}(x,t),
		\end{cases}
	\end{equation*}
and $u=\bar{u}-\underline{u}$. Then $u$ satisfies 
\begin{equation*}
	\begin{cases}
		u_{t}-\Delta_{{\Omega}} u+ \hat{c}(x,t) u(x,t) \ge 0,~~&(x,t)\in \Omega\times (0,T] ,\\
		\mathcal{B} u{\ge} 0,~& (x,t)\in \partial \Omega\times(0,T],\\
		u(x,0)\ge 0,~&x\in \Omega.
	\end{cases}
\end{equation*}
Clearly, $\hat{c}(x,t)$ is bounded. It follows from Theorem \ref{max} that $u(x,t)\ge 0$ for $x\in \bar{\Omega}$ and $t\in [0,T]$. Furthermore, if $\bar{u}(x,0)\not\equiv \underline{u}(x,0)$, then $u(x,0)\not\equiv 0$, $x\in \Omega$. Thus, we apply Theorem \ref{sm} to deduce that 
\begin{equation*}
	u(x,t)>0~\text{in}~\Omega \times (0,T]. 
\end{equation*} 
\end{proof}

\begin{Theorem}\label{sxj1}
	Let $\overline{u}$, $\underline{u}$ be an upper and a lower solution of \eqref{1s}, respectively,  $g:\partial\Omega\times [0,T]\to \mathbb{R}$ belong to $C^{0}(\partial{\Omega} \times[0,T]) $ and $\psi:\Omega\to \mathbb{R}$ be given. If 
	\begin{equation}\label{c}
			f(x,t,z)\text{ is continuous with respect to}~ (t,z)\in[0,T]\times[\underline{\eta},\bar{\eta}]
	\end{equation}
 for every $x\in\Omega$, $f$ satisfies the Lipschitz condition in $u\in [\underline{\eta},\bar{\eta}]$, i.e., there exists a positive constant $M$ such that 
	\begin{equation}\label{L}
		|f(x,t,u)-f(x,t,v)|\le M |u-v|
	\end{equation}
for all $(x,t)\in \Omega\times(0,T],~u,v\in [\underline{\eta},\bar{\eta}]$, 
  then \eqref{1s} admits a unique solution $ u\in C^{0}(\overline{\Omega}_{T}) \cap C^{1}({\Omega }\times[0,T]) $ satisfying $\underline{u}\le u \le \overline{u}$ on $\overline{\Omega}_{T}$, where $\overline{\Omega}_{T}:=\bar{\Omega }\times [0,T],\; \underline{\eta}=\min\limits_{\overline{\Omega}_{T}} \min \{\overline{u}, \underline{u}\}$ and $\overline{\eta}=\max\limits_{\overline{\Omega}_{T}} \max \{\overline{u}, \underline{u} \}$.
\end{Theorem}
 \begin{proof}
If we define 
 	\begin{equation*}\label{dy}
 		<\underline{u},\bar{u}>=\{u\in C^{0}(\overline{\Omega}_{T}) | \underline{u}\le u \le \bar{u}~ \text{in}~ \overline{\Omega}_{T}\},
 	\end{equation*}	
 then for any given $v\in <\underline{u},\bar{u}>$,  we see that $v\in C^{0}(\Omega \times [0,T])$. Combining this with \eqref{c}, we deduce that
 	\begin{equation*}
 	H(x,t):=	f(x,t,v(x,t))+Mv(x,t) \in C^{0}(\Omega \times [0,T]) . 
 	\end{equation*}
Therefore, the linear problem 
 		\begin{equation*}\label{}
 		\left\{\begin{array}{l}
 			u_{t} (x, t)-\Delta_{\Omega} u(x, t)+Mu(x,t)=H(x, t),~ x \in \Omega,~ t \in(0, T], \\
 			\mathcal{B}u(x, t)=g(x, t), \quad x \in \partial \Omega, \quad t \in[0, T], \\
 			u(x, 0)=\psi(x), \quad x \in \Omega
 		\end{array}\right.
 	\end{equation*}
 admits a unique solution $u(x,t)\in C^{0}({\bar\Omega }\times[0,T]) \cap C^{1}({\Omega }\times[0,T]) $ by Theorem \ref{N}.

Define $\mathcal{F}:<\underline{u},\bar{u}>\to X_{T}:=C^{0}(\bar{\Omega}\times [0,T])$ by $u=\mathcal{F}v$ and
\begin{equation*}
	\underline{u}_{0}=\underline{u},~ \underline{u}_{k}=\mathcal{F}\underline{u}_{k-1},~ \bar{u}_{0}=\bar{u},~ \bar{u}_{k}=\mathcal{F}\bar{u}_{k-1},~ k=1,2,\cdots.
\end{equation*}
Then we prove that 
$$\underline{u}\le \underline{u}_{1}\le\cdots \le\underline{u}_{k}\le \bar{u}_{k}\le\cdots\le\bar{u}_{1}\le \bar{u}\text{~for ~all~} k\ge 1.$$
For any $u, v\in<\underline{u},\bar{u}>$ satisfying $u\le v$ on $\bar{\Omega}\times[0,T]$. Set $w_1 = \mathcal{F} u$ $w_2 = \mathcal{F} v$ and $w=w_2 - w_1$. In view of \eqref{L}, then $w$ satisfies 
\begin{equation*}\label{}
	\left\{\begin{array}{l}
		w_{t} -\Delta_{\Omega} w+Mw= f(x,t,v)-f(x,t,u)+M(v-u)\ge 0~ in~ \Omega\times(0,T], \\
		\mathcal{B}w = 0~\text{ on}~ \partial{\Omega}\times[0,T], \\
		w(x, 0)= 0 ~\text{ on}~ \Omega.
	\end{array}\right.
\end{equation*}
Therefore, we can apply Theorem \ref{max} to conclude that $w\ge 0$ on $\bar{\Omega}\times [0,T]$. This implies that $w_1=\mathcal{F}u\le\mathcal{F} v=w_2$ on $\bar{\Omega}\times[0,T]$. It's easily seen that $\bar{u}\ge \underline{u}$ in $\bar{\Omega} \times[0,T]$ by Theorem \ref{o} and hence that $\underline{u}_{k}\le \bar{u}_{k}$ in $\bar{\Omega }\times (0,T]$ for any integer $k\ge 1$ with induction. 

Now, we assert that $\underline{u}\le \underline{u}_{1}$. Let $z=\underline{u}_{1}-\underline{u}$. Then we use  Theorem \ref{max} to the following problem
 \begin{equation*}\label{}
 	\left\{\begin{array}{l}
 	z_{t} (x, t)-\Delta_{\Omega} z(x, t)+M z \ge 0,~ x \in \Omega,~ t \in(0, T], \\
 		\mathcal{B}z(x, t)=0, \quad x \in \partial \Omega, \quad t \in[0, T], \\
 		z(x, 0)=0, \quad x \in \Omega
 	\end{array}\right.
 \end{equation*}
obtain that $\underline{u}\le \underline{u}_{1}$ on $\bar{\Omega}\times[0,T]$.
 Suppose that $\underline{u}_{k+1}\ge \underline{u}_{k}$ on $\bar{\Omega}\times[0,T]$. Denote $U:=\underline{u}_{k+2}-\underline{u}_{k+1}$. It follows from \eqref{L} that $U$ satisfies 
\begin{equation*}\label{}
	\left\{\begin{array}{l}
		U_{t} (x, t)-\Delta_{\Omega} U(x, t)+MU(x,t)= f(x,t,\underline{u}_{k+1})-f(x,t,\underline{u}_{k})+M(\underline{u}_{k+1}-\underline{u}_{k}) \ge 0~ \\
	\qquad\qquad\qquad\qquad\qquad\qquad\qquad\qquad\qquad\qquad\qquad\qquad\qquad\qquad\qquad	in~ \Omega\times(0,T], \\
		\mathcal{B}U \ge 0~\text{ on}~ \partial{\Omega}\times [0,T], \\
		U(x, 0)\ge 0 ~\text{ on}~ \Omega.
	\end{array}\right.
\end{equation*}
We use Theorem \ref{max} to deduce that $\underline{u}_{k+2}\ge \underline{u}_{k+1}$ on $\bar{\Omega}\times[0,T]$.
Therefore, by induction, we see that $\underline{u}_{k+1}\ge \underline{u}_{k}$ on $\bar{\Omega}\times[0,T]$ for all $k\ge 1$. Similarly, since $\bar{u}_{k}$ satisfies 
\begin{equation}\label{j}
	\left\{\begin{array}{l}
		(\bar{u}_{k})_{t} -\Delta_{\Omega} \bar{u}_{k}+M\bar{u}_{k}=f(x,t,\bar{u}_{k-1})+M\bar{u}_{k-1}(x,t),~ x \in \Omega,~ t \in(0, T], \\
		\mathcal{B}\bar{u}_{k}(x, t)=g(x, t), \quad x \in \partial \Omega, \quad t \in[0, T], \\
		\bar{u}_{k}(x, 0)=\psi(x), \quad x \in \Omega,
	\end{array}\right.
\end{equation}
we use the similar arguments of above to show that 
$\bar{u}_{k+1}\le \bar{u}_{k}$ on $\bar{\Omega}\times[0,T]$ for all $k\ge 1$.

Next, we can define 
\begin{equation*}
	\underline{u}(x,t):=\lim\limits_{k\to +\infty} \underline{u}_{k}(x,t),~\bar{u}(x,t):=\lim\limits_{k\to +\infty} \bar{u}_{k}(x,t)
\end{equation*}
For any given $x\in\Omega$ and any $(t_1,t_2)\subset [0,T]$, it follows from the first equality in \eqref{j} that
\begin{equation}\label{j2}
	\bar{u}_{k}(x,t_2)-\bar{u}_{k}(x,t_1)=\int_{t_1}^{t_2}\left( \Delta_{{\Omega}} \bar{u}_{k}-M\bar{u}_{k}+f(x,t,\bar{u}_{k-1})+M\bar{u}_{k-1} \right)dt
\end{equation}
As $k\to +\infty$,  \eqref{j2} becomes the following equation
\begin{equation*}\label{}
	\bar{u}(x,t_2)-\bar{u}(x,t_1)=\int_{t_1}^{t_2} \left(\Delta_{{\Omega}} \bar{u}-M\bar{u}+f(x,t,\bar{u})+M\bar{u}~ \right)dt
\end{equation*}
which implies that 
\begin{equation*}
	\bar{u}_{t} -\Delta_{\Omega} \bar{u}+M\bar{u}=f(x,t,\bar{u})+M\bar{u}(x,t)~\text{for}~x\in \Omega,~t\in (0,T].
\end{equation*}
Letting $k\to\infty$ in the second and third equalities of \eqref{j}, we deduce that 
\begin{equation*}\label{}
	\left\{\begin{array}{l}
		\mathcal{B}\bar{u}(x, t)=g(x, t), \quad x \in \partial \Omega, \quad t \in[0, T], \\
		\bar{u}(x, 0)=\psi(x), \quad x \in \Omega
	\end{array}\right.
\end{equation*}
Therefore, $\bar{u}\in C^{1}(\Omega\times[0,T])\cap C^{0}(\bar{\Omega}\times[0,T])$ satisfies \eqref{1s}.
Similarly, we can show that $\underline{u}\in C^{1}(\Omega\times[0,T])\cap C^{0}(\bar{\Omega}\times[0,T])$ satisfies \eqref{1s}.
 \end{proof}

\subsubsection{The parabolic equation with initial value} In this subsection, we establish the upper and lower solutions method for the problem \eqref{34}.

\begin{Definition}
	Suppose that $u(x,t)\in C^{1}(V \times [0,T] )$ satisfies 
	\begin{equation*}
		\begin{cases}
			u_{t}-\Delta_{V} u \ge (\le) f(x,t,u),~x\in V \times (0,T], \\
			u(x,0)\ge(\le) u_{0} (x),~x\in V.
		\end{cases}	
	\end{equation*}
	Then $u$ is called a supersolution (a subsolution) of the problem \eqref{34}.
\end{Definition}

\begin{Theorem}{\rm (Ordering of upper and lower solutions)} \label{oc}
	Let $\bar{u}$, $\underline{u}$ be an upper and a lower solution of \eqref{34}, respectively, and $u_0:V \to \mathbb{R}$ be given. Denote $\sigma=\inf\limits_{V \times[0,T ]}\min\{\bar{u},\underline{u} \}$ and $A=\sup\limits_{V \times[0,T ]} \max\{\bar{u},\underline{u} \}$. If $f$ satisfies the Lipschitz condition in $u\in [\sigma,A]$, i.e., there exists a constant $M>0$ such that \begin{equation}\label{4.51}
		|f(x,t,u)-f(x,t,v)| \le M |u-v|,~\forall~(x,t)\in V \times(0,T ],~u,v\in [\sigma,A], 
	\end{equation}then $\bar{u}\ge \underline{u}$ in $V \times[0,T ]$. Moreover, $\bar{u}>\underline{u}$ in $V \times (0,T ]$ if $\bar{u}(x,0)\not\equiv \underline{u}(x,0)$.
\end{Theorem}
\begin{proof}
If we define 
	\begin{equation*}
		\hat{c}(x,t)=\begin{cases}
			M,~\text{if}~\bar{u}(x,t)\ge \underline{u}(x,t), \\
			-M,~\text{if}~\bar{u}(x,t)< \underline{u}(x,t),
		\end{cases}
	\end{equation*}
	and $u=\bar{u}-\underline{u}$. Then $u$ satisfies 
	\begin{equation*}
		\begin{cases}
			u_{t}-\Delta_{{V}} u+ \hat{c}(x,t) u(x,t) \ge 0,~~&(x,t)\in V\times (0,T] ,\\
			u(x,0)\ge 0,~&x\in V.
		\end{cases}
	\end{equation*}
It follows from Theorem \ref{maxc} that $u(x,t)\ge 0$ for $x\in V$ and $t\in [0,T]$. Moreover, if $\bar{u}(x,0)\not\equiv \underline{u}(x,0)$, then we have $u(x,0)\not\equiv 0$, $x\in V$. Then, we use Theorem \ref{smc} to conclude that 
$
	u(x,t)>0~\text{in}~V \times (0,T]. 
	$ 
\end{proof}

\begin{Theorem}\label{sxc}
	Let $\overline{u}$, $\underline{u}$ be an upper and a lower solution of \eqref{34} respectively, and  $u_0:V \to \mathbb{R}$ be given. If 
	\begin{equation}\label{c1}
		f(x,t,z)\text{ is continuous with respect to}~ (t,z)\in[0,T]\times[\underline{\eta},\bar{\eta}]
	\end{equation}
	for every $x\in V$, $f$ satisfies the Lipschitz condition in $u\in [\underline{\eta},\bar{\eta}]$, i.e., there exists a positive constant $M$ such that 
	\begin{equation}\label{L2}
		|f(x,t,u)-f(x,t,v)|\le M |u-v|
	\end{equation}
	for all $(x,t)\in V \times(0,T],~u,v\in [\underline{\eta},\bar{\eta}]$, 
	then \eqref{34} admits a unique solution $u\in C^{1}(V\times[0,T])$ satisfying $\underline{u}\le u \le \overline{u}$ on $V\times[0,T]$, where $\underline{\eta}=\min\limits_{V\times[0,T]} \min \{\overline{u}, \underline{u}\}$ and $\overline{\eta}=\max\limits_{V\times[0,T]} \max \{\overline{u}, \underline{u} \}$.
\end{Theorem}
\begin{proof}
	Define 
	\begin{equation*}\label{20220809}
		<\underline{u},\bar{u}>=\{u\in C^{0}(V \times[0,T]) | {\underline{u}\le u\le \bar{u}}~ \text{in}~ V \times[0,T] \}.
	\end{equation*} 
For any given $v\in <\underline{u},\bar{u}>$, we apply \eqref{c1} to obtain that
	$$
	H(x,t):=	f(x,t,v(x,t))+Mv(x,t) \in C^{0}(V \times [0,T]) . 
	$$
It follows from Theorem \ref{3.3} that the linear problem 
	\begin{equation*}\label{}
		\left\{\begin{array}{l}
			u_{t} (x, t)-\Delta_{V} u(x, t)+Mu(x,t)=H(x, t),~ x \in V,~ t \in(0, T], \\
			u(x, 0)=u_0(x), \quad x \in V
		\end{array}\right.
	\end{equation*}
	admits a unique solution $u(x,t)$.

	We now define $\mathcal{F}:<\underline{u},\bar{u}>\to X_{T}:=C^{0}(V\times [0,T])$ by $u=\mathcal{F}v$ and
\begin{equation*}
		\underline{u}_{0}=\underline{u},~ \underline{u}_{k}=\mathcal{F}\underline{u}_{k-1},~ \bar{u}_{0}=\bar{u},~ \bar{u}_{k}=\mathcal{F}\bar{u}_{k-1},~ k=1,2,\cdots.
\end{equation*}
Next, we claim that 
	\begin{equation}\label{x1}
		\underline{u}\le \underline{u}_{1}\le\cdots \le\underline{u}_{k}\le \bar{u}_{k}\le\cdots\le\bar{u}_{1}\le \bar{u}~\text{on}~V\text{~for ~all~} k\ge 1.
	\end{equation}
	
	For any $u, v\in<\underline{u},\bar{u}>$ satisfying $u\le v$ on $V\times[0,T]$. Let $w_1 = \mathcal{F} u$, $w_2 = \mathcal{F} v$ and $w=w_2 - w_1$. Then we can find $w$ satisfies the following problem
	\begin{equation*}\label{}
		\left\{\begin{array}{l}
			w_{t} -\Delta_{V} w+Mw= f(x,t,v)-f(x,t,u)+M(v-u)\ge 0~ in~ V\times(0,T], \\
			w(x, 0)= 0 ~\text{ on}~ V.
		\end{array}\right.
	\end{equation*}
It follows from Theorem \ref{maxc} that $w\ge 0$ on $V\times [0,T]$. This implies that $w_1=\mathcal{F}u\le\mathcal{F} v=w_2$ on $V\times[0,T]$. By Theorem \ref{oc}, we have $\bar{u}\ge \underline{u}$ in $V\times[0,T]$. It's easily seen that $\underline{u}_{k}\le \bar{u}_{k}$ in ${V }\times (0,T]$ for any integer $k\ge 1$ by induction. 
	
We shall show that $\underline{u}_{k+1}\ge \underline{u}_{k}$ on $V\times[0,T]$ for all $k\ge 1$ by induction. Let $z=\underline{u}_{1}-\underline{u}$; then 
	\begin{equation*}\label{}
		\left\{\begin{array}{l}
			z_{t} (x, t)-\Delta_{V} z(x, t)+M z \ge 0,~ x \in V,~ t \in(0, T], \\
			z(x, 0)=0, \quad x \in V.
		\end{array}\right.
	\end{equation*}
In view of Theorem \ref{maxc}, we see that $\underline{u}\le \underline{u}_{1}$ on $V\times[0,T]$.
We assume that $\underline{u}_{k+1}\ge \underline{u}_{k}$ on $V\times[0,T]$. Denote $U:=\underline{u}_{k+2}-\underline{u}_{k+1}$. Then we apply Theorem \ref{maxc}
	\begin{equation*}\label{}
		\left\{\begin{array}{l}
			U_{t} (x, t)-\Delta_{V} U(x, t)+MU(x,t)= f(x,t,\underline{u}_{k+1})-f(x,t,\underline{u}_{k})+M(\underline{u}_{k+1}-\underline{u}_{k}) \ge 0~ \\
			\qquad\qquad\qquad\qquad\qquad\qquad\qquad\qquad\qquad\qquad\qquad\qquad\qquad\qquad\qquad	in~ V\times(0,T], \\
			U(x, 0)\ge 0 ~\text{ on}~ V.
		\end{array}\right.
	\end{equation*}
to conclude that $\underline{u}_{k+2}\ge \underline{u}_{k+1}$ on $V\times[0,T]$.
	Therefore, by induction, we see that $\underline{u}_{k+1}\ge \underline{u}_{k}$ on $V\times[0,T]$ for all $k\ge 1$. Similarly, since $\bar{u}_{k}$ satisfies 
	\begin{equation}\label{jc}
		\left\{\begin{array}{l}
			(\bar{u}_{k})_{t} -\Delta_{V} \bar{u}_{k}+M\bar{u}_{k}=f(x,t,\bar{u}_{k-1})+M\bar{u}_{k-1}(x,t),~ x \in V,~ t \in(0, T], \\
			\bar{u}_{k}(x, 0)=u_0(x), \quad x \in V.
		\end{array}\right.
	\end{equation}
We argument similar to those shown above imply via induction that 
	$\bar{u}_{k+1}\le \bar{u}_{k}$ on $V\times[0,T]$ for all $k\ge 1$.
Thus, \eqref{x1} holds.	
	
	Therefore, we can define 
	$
	\underline{u}(x,t):=\lim\limits_{k\to +\infty} \underline{u}_{k}(x,t),~\bar{u}(x,t):=\lim\limits_{k\to +\infty} \bar{u}_{k}(x,t).
	$
For any given $x\in V$ and any $(t_1,t_2)\subset [0,T]$, it follows from the first equality in \eqref{jc} that
	\begin{equation}\label{j2c}
		\bar{u}_{k}(x,t_2)-\bar{u}_{k}(x,t_1)=\int_{t_1}^{t_2}\left( \Delta_{{V}} \bar{u}_{k}-M\bar{u}_{k}+f(x,t,\bar{u}_{k-1})+M\bar{u}_{k-1} \right)dt.
	\end{equation}
As $k\to +\infty$ in \eqref{j2c}, $\overline{u}$ satisfies 
	$$
	\bar{u}(x,t_2)-\bar{u}(x,t_1)=\int_{t_1}^{t_2}\left( \Delta_{{V}} \bar{u}-M\bar{u}+f(x,t,\bar{u})+M\bar{u}~ \right)dt.
	$$
	This implies that 
	\begin{equation*}
		\bar{u}_{t} -\Delta_{V} \bar{u}+M\bar{u}=f(x,t,\bar{u})+M\bar{u}(x,t)~\text{for}~x\in V,~t\in [0,T].
	\end{equation*}
As $k\to\infty$ in the second equality of \eqref{jc}, we obtain 
	$	
	\bar{u}(x, 0)=u_0(x), ~ x \in V	.
	$
In other word, $\bar{u}\in C^{1}(V\times[0,T])$ satisfies \eqref{34}. Similarly, $\underline{u}\in C^{1}(V\times[0,T])$ satisfies \eqref{34}.
\end{proof}

\subsection{The upper and lower solutions method for elliptic problems}
\subsubsection{The elliptic equation with boundary}
In this subsection, we establish the upper and lower solutions method for the problem \eqref{t}.

 Define 
\begin{equation*}
	H_{0}^{1}(\bar{\Omega}):=\{u\in H^{1}(\Omega) | u(x)=0~\text{for}~x\in \partial{\Omega}\}
\end{equation*}
and 
\begin{equation*}
	||u||_{\bar{\Omega }}:=\left( \int\limits_{\bar{\Omega}} \left(|\nabla u|^{2}+u^{2}\right)d\mu\right) ^{\frac{1}{2}},~\forall ~u\in H_{0}^{1}(\bar{\Omega}). 
\end{equation*}
It is easy to check that $\left( H_{0}^{1}(\bar{\Omega}), \| \cdot\|_{\bar{\Omega}}\right) $ is a Banach space.
\begin{Definition}
	A function $u:\bar{\Omega}\to \mathbb{R}$ is called an upper (a lower) solution of \eqref{t}, if 
	\begin{equation*}\label{}
		\begin{cases}
			-\Delta_{{\Omega}} u\ge (\le)~ f(x,u),~&x\in \Omega,  \\
			\mathcal{B} [u]\ge (\le)~ \phi(x),~&x\in \partial{\Omega}.
		\end{cases}
	\end{equation*}
\end{Definition}

 \begin{Theorem}\label{44}
 	Let $\bar{u}$ and $\underline{u}$ be an upper and a lower solution of \eqref{t} respectively and $\phi:\bar{\Omega}\to \mathbb{R}$ be given. Denote $\underline{c}:=\min\limits_{\bar{\Omega}}\underline{u}$, $\overline{c}:=\max\limits_{\bar{\Omega}}\bar{u}$. Suppose that $\underline{u}\le \bar{u}$ on $\bar{\Omega}$ and that there exists constant $M>0$ such that 
 	\begin{equation}\label{L3}
 		f(x,u)-f(x,v)\ge -M(u-v)~\forall x\in \overline{\Omega},~u,v\in [\underline{c},\overline{c}] .
 	\end{equation}
 Then there exist sequences $\{\underline{u}_{m} \}_{m=1}^{+\infty}$ and $\{ \overline{u}_{m}  \}_{m=1}^{+\infty}$ such that $\{\underline{u}_{m} \}_{m=1}^{+\infty}$ is monotone nondecreasing with respect to $m$, $\{ \overline{u}_{m}  \}_{m=1}^{+\infty}$ is monotone nonincreasing with respect to $m$, 
 \begin{equation*}
 	\underline{u}\le \underline{u}_{m}\le \overline{u}_{m} \le  \overline{u}~\forall m\ge 1,
 \end{equation*} 
and 
\begin{equation*}
	\underline{u}_{m} \to \tilde{u},~\overline{u}_{m} \to \hat{u}~\text{as}~m\to +\infty,
\end{equation*}
where $\tilde{u}$ and $\hat{u}$ are the minimal and maximal solutions of \eqref{t} in the following sense: if $u$ is a solution to \eqref{t} satisfying $\underline{u}\le u\le \bar{u}$ on $\bar{\Omega}$, then $ \tilde{u}\le u \le \hat{u}$ on $\bar{\Omega}$.
 \end{Theorem}
\begin{proof}
	For any function $v\in <\underline{u},\bar{u}>:=\{v\in R\bar{\Omega}~ |~ \underline{u}(x)\le v(x) \le \bar{u}(x)~\text{on}~\bar{\Omega}  \}$. Considering the following problem
	\begin{equation}\label{32}
		\begin{cases}
			-\Delta_{{\Omega}} u+ Mu= f(x,v)+Mv,~x\in\Omega,\\
			\mathcal{B} u= \phi,~x\in \partial{\Omega}.
		\end{cases}
	\end{equation}

Now, we are going to prove that \eqref{32} has a unique solution. We first deal with the case that $\mathcal{B}u=u$. 
Suppose that $u$ satisfies \eqref{32}.
Let $w=u-\phi$; then $w$ satisfies 
\begin{equation*}
	\begin{cases}
		-\Delta_{{\Omega}} w+Mw= f(x,v)+Mv-M\phi+ \Delta_{{\Omega}} \phi=:g(x),~&x\in \Omega, \\
		w=0, &x\in \partial{\Omega}.
	\end{cases}
\end{equation*}
If we define
\begin{equation}\label{561}
	B(u,v):=\int\limits_{\bar\Omega}  \left(\Gamma(u,v)+Muv\right) d\mu,~u,v\in H_{0}^{1}(\bar{\Omega}).
\end{equation}
It follows from Cauchy Schwartz inequality and H$\ddot{\text{o}}$lder inequality to deduce that
\begin{equation*}
	\begin{aligned}
		|B(u, v)| 
		& \leq \int\limits_{\bar{\Omega}} \left[\sum_{y\in \bar{\Omega}:y \sim x} \frac{\omega_{x y}}{2 \mu(x)}(u(y)-u(x))(v(y)-v(x))\right] d \mu+M \left(\int\limits_{\bar{\Omega}} u^{2} d \mu \right)^{\frac{1}{2}}\left(\int\limits_{\bar{\Omega}} v^{2} d \mu\right)^{\frac{1}{2}} \\		
		&=\int\limits_{\bar{\Omega}}[\Gamma(u, u)]^{\frac{1}{2}}[\Gamma(v, v)]^{\frac{1}{2}} d \mu+ M ||u||_{L^2(\bar{\Omega})}||v||_{L^2(\bar{\Omega})} \\
		& \leq\left(\int\limits_{\bar{\Omega}}|\nabla u|^{2} d \mu\right)^{\frac{1}{2}}\left(\int_{\bar{\Omega}}|\nabla v|^{2} d \mu \right)^{\frac{1}{2}}+M||u||_{L^2(\bar{\Omega})}||v||_{L^2(\bar{\Omega})}\\
		&\le (1+M)\left[ \left(\int\limits_{\bar{\Omega}}|\nabla u|^{2} d \mu\right)^{\frac{1}{2}}\left(\int_{\bar{\Omega}}|\nabla v|^{2} d \mu \right)^{\frac{1}{2}}+||u||_{L^2(\bar{\Omega})}||v||_{L^2(\bar{\Omega})}\right] \\
		&\le (1+M)\|u \|_{H_{0}^{1}(\bar{\Omega})} \|v \|_{H_{0}^{1}(\bar{\Omega})}.
	\end{aligned}
\end{equation*}
In the light of \eqref{561} , we have 
\begin{equation*}\label{}
	\begin{aligned}
			B(u,u)\ge \int\limits_{\bar{\Omega}} \left( |\nabla u|^2 + M u^2\right) d\mu \ge \min\{1,M \} \int\limits_{\bar{\Omega}} \left(|\nabla u|^{2} + u^{2}\right) d\mu
			=\min\{1, M \} \|u \|_{H_{0}^{1}(\bar{\Omega})}.
	\end{aligned}
\end{equation*}
It is easy to check that $B$ is bilinear. Then, by the Lax-Milgram theorem, for any function $g:$ $\bar{\Omega}\to \mathbb{R}$, there exists a unique function $U\in H_{0}^{1}(\bar{\Omega})$ such that $
	B(U,v)=\int_{\bar\Omega} gv d\mu~\text{~for~any~}v\in H_{0}^{1}(\bar{\Omega}).
$
This is equivalent to 
\begin{equation*}\label{}
	\begin{cases}
		-\Delta_{{\Omega}} U+M U= g~\text{in}~\Omega,\\
		U=0~\text{ on}~\partial \Omega.
	\end{cases}
\end{equation*}
That is to say, \eqref{32} admits a unique solution $U$. 
 
In the following, we consider the case that $\mathcal{B}u= \frac{\partial u}{\partial_{\Omega} n}$. Clearly, there exists a function $I(x)$ such that 
$
 -	\int\limits_{\Omega} I(x) d\mu = \int\limits_{\partial\Omega} \phi(x) d \mu.
$
Then, we use \cite[Theorem 3.8]{CB} to conclude that the elliptic problem 
\begin{equation*}
	\begin{cases} -\Delta_{\Omega} Z(x)=I(x) & x \in \Omega, \\
		\frac{\partial Z}{\partial_{\Omega} n}(x)=\phi(x)  , & x \in \partial{\Omega} \end{cases}
\end{equation*}
admits a unique solution $Z(x)$.
 Suppose that $u$ satisfies \eqref{32}. Let $w=u-Z$. It's easy to check that $w$ satisfies 
 \begin{equation}\label{42}
 	\begin{cases} -\Delta_{\Omega} w(x)+Mw=f(x,v)+Mv-I-MZ=:F(x), & x \in \Omega, \\
 		\frac{\partial w}{\partial_{\Omega} n}(x)=0  , & x \in \partial{\Omega}. \end{cases}
 \end{equation}
We now consider the following expansions
$$
	w(x)=\sum_{i=1}^{n} c_{i} \Phi_{i}(x),~
F(x)=\sum_{i=1}^{n} F_{i} \Phi_{i}(x),~x\in \Omega,
$$
where 
$
	c_i=\sum\limits_{x \in \Omega} w(x)\Phi_{i} (x),
	F_i=\sum\limits_{x \in \Omega} F(x)\Phi_{i} (x).
$
As we substitute these expansions into \eqref{42}, then it follows that
$$
	c_{i}=\frac{F_{i}}{K_{i}+ M},~i=1,\cdots,n.
$$
Therefore, we can conclude that 
\begin{equation}\label{q}
	w(x)=\sum_{i=1}^{n} \frac{F_{i}(x)}{K_{i}+M} \Phi_{i}(x)
\end{equation}
is a solution of \eqref{42}. Applying Theorems \ref{25}, we deduce that $w(x)$ defined by \eqref{q} is the unique solution to \eqref{42}. Thus $Z=u-w$ is the unique solution to \eqref{32}.
So that we can define an operator $T$ by $u=Tv$. 

Next, we first show that $T$ is monotone. For any $v,w\in <\underline{u},\bar{u}>$ and $v\le w$. Let $z=Tw$, $u=Tv$; then $h=z-u$ satisfies 
\begin{equation*}
	\begin{cases}
		-\Delta_{{\Omega}} h+Mh=\left[ f(x,w)-f(x,v)  \right] + M(w-v)\ge 0,~&x\in \Omega, \\
		\mathcal{B} h =0 ,&x\in\partial{\Omega}.
	\end{cases}
\end{equation*}
We use Theorem \ref{25} to deduce that $h\ge 0$ on $\bar{\Omega}$ and hence that 
$
	u=Tv\le z=Tw.
$

Define 
$$
	\underline{u}_{1}=T\underline{u},~	\underline{u}_{m+1}=T\underline{u}_{m},
	\overline{u}_{1}=T\overline{u},~	\overline{u}_{m+1}=T\overline{u}_{m}
$$
for $m=1,2,3,\cdots$. It follows from the monotonicity of $T$ to deduce that $\underline{u}_{1}\le \overline{u}_{1}$. Denote $v=\overline{u}-\overline{u}_{1}$. Then $v$ satisfies 
\begin{equation*}
	\begin{cases}
	-\Delta_{{\Omega}} v +Mv \ge f(x,\bar{u})+M\bar{u}- [f(x,\bar{u})+M\bar{u}]=0,~x\in \Omega,\\
	\mathcal{B} v \ge 0,~x\in \partial{\Omega}.
	\end{cases}
\end{equation*}
Due to Theorem \ref{25}, we obtain that
$
	v\ge 0~\text{ on}~\bar{\Omega},
$
i.e, $\bar{u} \ge \bar{u}_{1}$. Similarly, we can show that $\underline{u}\le \underline{u}_{1}$ on $\bar{\Omega}$. Thus, by induction, we have
\begin{equation*}
	\underline{u}\le \underline{u}_{1}\le \cdots \le \underline{u}_{m} \le \overline{u}_{m}\le \cdots \le \overline{u}_{1} \le \bar{u}~\text{ on}~\bar{\Omega}
\end{equation*}
for all $m\ge 1$.

Therefore, we can define 
$
	\tilde{u}(x):=\lim\limits_{m\to +\infty} \underline{u}_{m}(x)
$
and
$
	\hat{u}(x):=\lim\limits_{m\to +\infty} \overline{u}_{m}(x).
$
Letting $m\to +\infty$ in the following elliptic problem
\begin{equation*}
	\begin{cases}
		(-\Delta_{{\Omega}}+M)\underline{u}_{m+1}=M \underline{u}_{m}+f(x,\underline{u}_{m}),~x\in \Omega,\\
		\mathcal{B}\underline{u}_{m+1}=\phi,~x\in \partial{\Omega}.
	\end{cases}
\end{equation*} 
we obtain
$
	\tilde{u}=T\tilde{u},~\hat{u}=T\hat{u}.
$
That is to say, $\tilde{u}$ and $\hat{u}$ are solutions of \eqref{t}. Set $\underline{u}\le u\le \overline{u}$ be a solution of \eqref{t}. Then we have $u=Tu$. In view of $T$ is nondecreasing, we see that 
$
\underline{u}_{1}=T \underline{u}\le Tu=u \le T\overline{u}=\overline{u}_{1},
$
and thus, by induction,
\begin{equation*}\label{gd}
\underline{u}_{m}=T\underline{u}_{m-1}\le	 Tu=u\le T\overline{u}_{m-1}=\overline{u} _{m}~\text{on}~\bar{\Omega}~\text{ for ~all~}m\ge 1.	
\end{equation*}
As $m\to \infty$ in \eqref{gd}, we obtain that  
$
\tilde{u}\le u\le \hat{u}~\text{on}~\bar{\Omega}.
$
\end{proof}

\begin{Remark}
	Theorem \ref{44} still holds with the assumption \eqref{L3} replaced with 
	\begin{equation*}\label{L4}
		|f(x,u)-f(x,v)|\le M |u-v|.
	\end{equation*}
	 
\end{Remark}

\subsubsection{The elliptic equation without boundary} In this subsection, we establish the upper and lower solutions method for the problem \eqref{tc}.
\begin{Definition}
	A function $u$ is called an upper (a lower) solution of \eqref{tc}, if 
	\begin{equation*}\label{}
			-\Delta_{{V}} u-\vec{b}\cdot\nabla u+c(x)u \ge (\le)~ f(x,u),~x\in V.
	\end{equation*}
\end{Definition}

\begin{Theorem}\label{44c}
	Let $\bar{u}$ and $\underline{u}$ be an upper and a lower solution of \eqref{tc} respectively. Denote $\underline{c}:=\min\limits_{V}\underline{u}$, $\overline{c}:=\max\limits_{V}\bar{u}$. Suppose that $\underline{u}\le \bar{u}$ on $V$ and $b(x)\ge 0$ on $V$ and that there exists constant $M^{'}>0$ such that 
	\begin{equation}\label{L5}
		|f(x,u)-f(x,v)|\le M^{'}|u-v|~\forall x\in V,~u,v\in [\underline{c},\overline{c}] .
	\end{equation}
	Then there exist sequences $\{\underline{u}_{m} \}_{m=1}^{+\infty}$ and $\{ \overline{u}_{m}  \}_{m=1}^{+\infty}$ such that $\{\underline{u}_{m} \}_{m=1}^{+\infty}$ is monotone nondecreasing with respect to $m$, $\{ \overline{u}_{m}  \}_{m=1}^{+\infty}$ is monotone nonincreasing with respect to $m$, 
$
		\underline{u}\le \underline{u}_{m}\le \overline{u}_{m} \le  \overline{u}~\forall m\ge 1,
$
	and 
$
		\underline{u}_{m} \to \tilde{u},~\overline{u}_m \to \hat{u}~\text{as}~m\to +\infty,
$
	where $\tilde{u}$ and $\hat{u}$ are the minimal and maximal solutions of \eqref{tc} in the following sense: if $u$ is a solution to \eqref{tc} satisfying $\underline{u}\le u\le \bar{u}$ on $\bar{\Omega}$, then $\tilde{u}\le u \le \hat{u}$ on $\bar{\Omega}$.
\end{Theorem}
\begin{proof}
Define 
 $$<\underline{u},\bar{u}>:=\{u\in RV | \underline{u}\le u \le \bar{u}~\text{ on}~V\}.$$	
 For any $v\in <\underline{u},\bar{u}>$, we choose a sufficiently large $M>M'$ such that $M+c_0-\frac{|\vec{b}|^{2}}{2}>1$ with $c_0 =\min\limits_{{V} } c(x)$ and  consider the following problem
	\begin{equation}\label{32c}
			-\Delta_{{V}} u-\vec{b}\cdot\nabla u+c(x)u +Mu= f(x,v)+Mv,~x\in V.
	\end{equation} 
 
Denote $E(x)=M+c(x)$. If we define a bilinear form $B[\cdot,\cdot]$ by
\begin{equation*}\label{561,}
	B(u,w):=\int\limits_{V}  \left[\Gamma(u,w)-\left( \vec{b}\cdot\nabla u\right)w +Euw\right] d\mu
\end{equation*}
for $~u,w\in H^{1}(V)$.
Then we use Cauchy Schwartz inequality and H$\ddot{\text{o}}$lder inequality to obtain that
\begin{equation*}
	\begin{aligned}
		|B(u, w)| & \leq \int\limits_{V}  \left[\Gamma(u, w)+\sqrt{\sum_{x \in V}b^{2}(x)} |\nabla u||w|+E  |u||w| \right]d \mu \\	
		&= ||\nabla u||_{L^2(V)}||\nabla w||_{L^2(V)}+\sqrt{\sum_{x \in V}b^{2}(x)}	||\nabla u||_{L^2(V)} || w||_{L^2(V)}
	+E || u ||_{L^2(V)} ||w||_{L^2(V)} \\
		&\le \left( 1+E+\sqrt{\sum_{x \in V}b^{2}(x)} \right) \|u \|_{H^{1}(V)} \|w \|_{H^{1}(V)}.
	\end{aligned}
\end{equation*}
It's obviously that 
$$
	B(u,u)=\int\limits_{V}  \left(|\nabla u|^2 + E u^2-u\vec{b}\cdot \nabla u \right) d\mu.
$$
If we choose $\epsilon$ small enough so that $\epsilon|b|=\frac{1}{2}$, by Cauchy's inequality, we have 

\begin{equation*}\label{}
	\begin{aligned}
		B(u,u)&\ge\int\limits_{V}\left(|\nabla u|^2 + E u^2-|\vec{b}|\nabla u||u|\right)d\mu\\
		&\ge\int\limits_{V}\left(|\nabla u|^2 + E u^2\right) d\mu-|\vec{b}|\left(\int_{V} \epsilon|\nabla u|^{2}+\frac{1}{4\epsilon} u^{2} d\mu\right)\\
		&\ge \int\limits_{V} \left[\frac{1}{2} |\nabla u|^2 +\left(M+c_0-\frac{|\vec{b}|}{4\epsilon} \right)u^{2}  \right]d\mu\\
		&\ge \min\{\frac{1}{2}, M+c_0-\frac{|\vec{b}|^{2}}{2} \} \int\limits_{V}\left( |\nabla u|^{2} + u^{2} \right)d\mu .
	\end{aligned}
\end{equation*}
 
  Denote $F(x):=f(x,v)+Mv$, we use Lax-Milgram theorem to find a unique function $u\in H^{1}(V)$ such that $$B(u,g)=\int\limits_{V}Fg d\mu\;\forall g\in H^{1}(V).$$ 
  This is equivalent to 
 the problem \eqref{32c} admits a unique solution $u$. Now, we can define an operator $T$ by $u=Tv$ and show that $T$ is monotone. For any $v,w\in <\underline{u},\bar{u}>$ satisfying $v\le w$. Let $z=Tw$, $u=Tv$; then by \eqref{L5}, we can find $h=z-u$ satisfies 
	\begin{equation}
			-\Delta_{V} h-\vec{b}\cdot\nabla h+\left( c(x)+ M\right)h=\left[ f(x,w)-f(x,v)  \right] + M(w-v)\ge 0,~x\in V.
	\end{equation}
It follows from Theorem \ref{27} that $h\ge 0$ on $V$ and hence that 
$
		u=Tv\le Tw=z.
$
	
	Define
$$
		\underline{u}_{1}=T\underline{u},~	\underline{u}_{m+1}=T\underline{u}_{m},
		\overline{u}_{1}=T\overline{u},~	\overline{u}_{m+1}=T\overline{u}_{m}
$$
	for $m=1,2,3,\cdots$. By the monotonicity of $T$, $\underline{u}_{1}\le \overline{u}_{1}$. Setting $v=\overline{u}-\overline{u}_{1}$; then we can find $v$ satisfies 
	\begin{equation*}
			-\Delta_{V} v-\vec{b}\cdot\nabla v+\left( c(x)+ M\right)v \ge f(x,\bar{u})+M\bar{u}- [f(x,\bar{u})+M\bar{u}]=0,~x\in V.
	\end{equation*}
Applying Theorem \ref{27}, we have
	$
		v\ge 0~\text{on}~V,
$
	i.e, $\bar{u} \ge \bar{u}_{1}$ on $V$. Similarly, $\underline{u}\le \underline{u}_{1}$ on $V$. Furthermore, by induction, we have
$$
		\underline{u}\le \underline{u}_{1}\le \cdots \le \underline{u}_{m} \le \overline{u}_{m}\le \cdots \le \overline{u}_{1} \le \bar{u}~\text{on}~V
$$
	for all $m\ge 1$.
	
Therefore, we can define 
$
		\tilde{u}(x):=\lim\limits_{m\to +\infty} \underline{u}_{m}(x)
$
	and
$
		\hat{u}(x):=\lim\limits_{m\to +\infty} \overline{u}_{m}(x).
$
Considering the elliptic problem
$$(-\Delta_{{V}}-\vec{b}\cdot\nabla+c(x)+M)\underline{u}_{m+1}=M \underline{u}_{m}+f(x,\underline{u}_{m}),$$ we obtain
$
		\tilde{u}=T\tilde{u},~\hat{u}=T\hat{u}
$ as $m\to +\infty$.
	Thus, $\tilde{u}$ and $\hat{u}$ are solutions of \eqref{tc}. Let $u\in <\underline{u},\bar{u}>$ be a solution of {\eqref{tc}}; it's obvious that $u$ satisfies $u=Tu$. Since $T$ is nondecreasing, we see that 
$
		 \underline{u}_{1}=T\underline{u} \le Tu=u \le T\overline{u}=\overline{u}_{1},
$
	and thus, by induction,
	\begin{equation}\label{c2}
	 \underline{u}_{m}=T\underline{u}_{m-1}\le Tu=u\le T\overline{u}_{m-1}=\overline{u} _{m}~\text{on}~V	
	\end{equation}
for all $m\ge 2$.
	Letting $m\to \infty$ in \eqref{c2}, we deduce that  
$
	\tilde{u}\le u\le \hat{u}~\text{on}~V.
$
\end{proof}

\section{The extinction and propagation of solutions for parabolic problems}
In this section, we study the monotonicity and convergence of solutions of the initial-boundary value problem \eqref{IB} and the initial value problem \eqref{5c}.
\subsection{Monotonicity and convergence}
\subsubsection {The parabolic equation of the initial-boundary problem}
In this subsection, we investigate the monotonicity and convergence of solutions of the problem
\begin{equation}\label{IB}
	\begin{cases}u_{t}-\Delta_{{\Omega}} u=f(x, u), & x \in \Omega, \quad 0<t\lesssim T, \\ 
	\mathcal{B}u=g(x), & x \in \partial \Omega,~ 0<t\lesssim T, \\ u(x, 0)=\varphi(x), & x \in \Omega,\end{cases}
\end{equation}
where  $g:\partial{\Omega} \to \mathbb{R}$ and $\varphi:\Omega:\to \mathbb{R}$ are given functions and $0<T\le \infty$.

In the following, the monotonicity of solutions to \eqref{IB} is proved when initial data $\varphi$ is a subsolution or supersolution to \eqref{t}.

\begin{Theorem}
	Assume that $f$ satisfies \eqref{4.5} for any $\sigma,\omega\in \mathbb{R}$. If the initial data $\varphi$ of \eqref{IB} is a subsolution (a supersolution) of \eqref{t}, and the problem \eqref{IB} admits a solution $u$ which is defined in $\bar{\Omega}\times [0,T]$ for some $T\le \infty$. Then $u$ is monotone nondecreasing (nonincreasing) in time $t$, i.e., $u(x,s)\le(\ge) u(x,t)$ for $0\le s<t\lesssim T$ and $x\in \bar{\Omega}$.
\end{Theorem}
\begin{proof}
We only deal with the case that $\varphi$ is a subsolution of \eqref{t}, and the case that $\varphi$ is a supersolution of \eqref{t} can be proved by a similar approach. 	It's easy to check that $\varphi$ is a lower solution of \eqref{IB}.  Applying Theorem \ref{o} and the locally Lipschitz continuous of $f$ in $u$, we deduce that $u(x,t)\ge \varphi(x)$ for  $x\in \bar{\Omega}$ and $ 0\le t \lesssim T$. For $0<\epsilon<T$, it's easy to check that $v(x,t)=u(x,\epsilon+t)$ is a solution of the problem
	\begin{equation*}
		\left\{\begin{array}{lll}
			v_{t}-\Delta_{{\Omega}} v=f(x, v), & x \in \Omega, & 0<t\lesssim T-\epsilon, \\
			\mathcal{B}v=g(x), & x \in \partial \Omega, & 0<t\lesssim T-\epsilon, \\
			v(x, 0)=u(x, \epsilon), & x \in \Omega.
		\end{array}\right.
	\end{equation*}
It follows from Theorem \ref{o} that $v(x,t)=u(x,\epsilon+t)\ge u(x,t)$ for $x\in \bar{\Omega}$, $0\le t\lesssim T-\epsilon$.
\end{proof}

Next, we study the long time behavior of solutions of \eqref{IB} which is monotone with respect to $t$.

\begin{Theorem}\label{45}
	Suppose that \eqref{IB} admits a solution $u(x,t)$, which is monotone nondecreasing (nonincreasing) with respect to $t$ for $t>0$ and $x\in \bar{\Omega}$. Then \eqref{t} has at least one solution $S(x)$, and 
	\begin{equation}\label{589}
		\lim\limits_{t\to +\infty} \max\limits_{x\in \bar{\Omega}}|u(x,t)-S(x)|=0.
	\end{equation}
\end{Theorem}
\begin{proof}
Note that $u(x,t)$ is monotone with respect to $t$, then we can find $v$ such that 
\begin{equation}\label{s}
		v(x)=\lim\limits_{t\to +\infty} u(x,t)~\text{for}~x\in \bar{\Omega}.
	\end{equation}
For any given function $\varphi(x): \Omega \to \mathbb{R}$ and any $T>0$, we multiply the first equation of \eqref{IB} by $\varphi$ and integrate on $\Omega\times [0,T]$ to obtain
\begin{equation}\label{591}
\frac{1}{T}	\int_{0}^{T}\int_{\Omega} \left(u_{t}\varphi- \varphi \Delta_{{\Omega}} u  \right)d\mu dt= \int_{0}^{T}\int_{\Omega} \frac{f(x,u)\varphi}{T} d\mu dt. 
\end{equation}
In view of \eqref{s} and the finiteness of $\Omega$ , we see that
\begin{equation*}
	\lim\limits_{T\to +\infty}\int_{0}^{T}\int_{\Omega} \frac{u_{t}\varphi}{T} d\mu dt=\int_{\Omega} \lim\limits_{T\to +\infty} \frac{u(x,T)-u(x,0)}{T} \varphi d\mu=0,
\end{equation*}	
\begin{equation*}
	\lim\limits_{T\to +\infty}\int_{0}^{T}\int_{\Omega} \frac{\varphi\Delta_{{\Omega}}u }{T} d\mu dt=\int_{\Omega} \lim\limits_{T\to +\infty} \int_{0}^{T}\frac{\Delta_{{\Omega}} u }{T}dt ~\varphi d\mu
	=\int_{\Omega} \Delta_{{\Omega}} v(x)\varphi d\mu,		
\end{equation*}
and
\begin{equation*}
		\lim\limits_{T\to +\infty}\int_{0}^{T}\int_{\Omega} \frac{f(x,u)\varphi }{T} d\mu dt=\int_{\Omega} \lim\limits_{T\to +\infty} \int_{0}^{T}\frac{f(x,u) }{T}dt ~\varphi d\mu
		=\int_{\Omega} f(x,v) \varphi d\mu.		
\end{equation*}
Thus, letting $T\to +\infty$ in \eqref{591}, we conclude that 
$$
	-\int_{\Omega} \Delta_{{\Omega}} v(x) \varphi(x) d\mu= \int_{\Omega} f(x,v) \varphi d\mu.
$$
It follows that
\begin{equation*}
	\begin{cases}
		-\Delta_{{\Omega}} v(x)= f(x,v),~&x\in \Omega, \\
		\mathcal{B}v=g(x),~&x\in\partial \Omega.
		\end{cases}
\end{equation*}
Therefore, we conclude that
$
	\lim\limits_{t\to +\infty} u(x,t)=v(x)~\text{for}~x\in \bar{\Omega}.
$
As $\bar{\Omega}$ is finite, we deduce that 
$$
	\lim\limits_{t\to +\infty} \max\limits_{x\in \bar{\Omega}}|u(x,t)-v(x)|=0.
$$ Taking $S(x)=v(x)$, then we obtain \eqref{589}.
\end{proof}

\subsubsection{The parabolic equation of the initial value}
In this subsection, we investigate the monotonicity and convergence of solutions of the following problem
\begin{equation}\label{5c}
	\begin{cases}u_{t}-\Delta_{V} u=f(x, u), & x \in V, \quad 0<t\lesssim T, \\ u(x, 0)=\varphi(x), & x \in V,\end{cases}
\end{equation}
where $0<T\le \infty$ and $\varphi:V\to \mathbb{R}$ are given function.

In the following, the monotonicity of solutions to \eqref{5c} is proved when initial data $\varphi$ is a subsolution or supersolution to \eqref{tc}.

\begin{Theorem}
	Assume that $f$ satisfies \eqref{4.51} for any $\sigma,A\in \mathbb{R}$. If the initial data $\varphi$ of \eqref{5c} is a subsolution (a supersolution) of \eqref{tc}, and the problem \eqref{5c} admits a solution $u$ which is defined in $V \times [0,T]$ for some $T\le \infty$. Then $u$ is monotone nondecreasing (nonincreasing) in time $t$, i.e., $u(x,s)\le(\ge) u(x,t)$ for $0\le s<t\lesssim T$ and $x\in V$.
\end{Theorem}
\begin{proof}
We only deal with the case that $\varphi$ is a subsolution of \eqref{tc}. For the case of supersolution $\varphi$ to \eqref{tc}, we can use analogy methods to achieve the conclusion.	It is easy to check that $\varphi$ is a lower solution of \eqref{5c}. We use Theorem \ref{oc} and  the locally Lipschitz continuous of $f$ in $u$ to obtain that $u(x,t)\ge \varphi(x)$ for $x\in V $ and $0\le t \lesssim T$. Let $0<\epsilon<T$. It is easy to check that $v(x,t)=u(x,\epsilon+t)$ is a solution of the problem
\begin{equation*}
		\left\{\begin{array}{lll}
			v_{t}-\Delta_{V} v=f(x, v), & x \in V, & 0<t\lesssim T-\epsilon, \\
			v(x, 0)=u(x, \epsilon), & x \in V .
		\end{array}\right.
\end{equation*}
	Thus, by Theorem \ref{oc}, we conclude that $v(x,t)=u(x,\epsilon+t)\ge u(x,t)$ for $x\in V$, $0< t\lesssim T-\epsilon$.
\end{proof}

Now, we study the long time behavior of solutions of \eqref{5c} which is monotone with respect to $t$.

\begin{Theorem}\label{45c}
	Suppose that \eqref{5c} admits a solution $u(x,t)$, which is monotone nondecreasing (nonincreasing) with respect to time $t$ for $t>0$ and $x\in V$. Then {\eqref{tc}} has at least one solution $e(x)$, and 
	\begin{equation}\label{101}
		\lim\limits_{t\to +\infty} \max\limits_{x\in V}|u(x,t)-e(x)|=0.
	\end{equation}
\end{Theorem}
\begin{proof}
As $u(x,t)$ is monotone with respect to $t$, we have 
\begin{equation}\label{sc}
		v(x):=\lim\limits_{t\to +\infty} u(x,t)~\text{for}~x\in V.
\end{equation}
	For any given function $\varphi(x): V \to \mathbb{R}$ and any $T>0$, we multiply the first equation of \eqref{5c} by $\varphi$ and integrate on $V\times[0,T]$, then 
	\begin{equation}\label{5}
		\frac{1}{T}	\int_{0}^{T}\int_{V}\left( u_{t}\varphi- \varphi \Delta_{{V}} u\right)  d\mu dt= \int_{0}^{T}\int_{V} \frac{f(x,u)\varphi}{T} d\mu dt. 
	\end{equation}
Thanks to \eqref{sc} and the finiteness of $V$, we have  
	\begin{equation*}\label{11c}
		\lim\limits_{T\to +\infty}\int_{0}^{T}\int_{V} \frac{u_{t}\varphi}{T} d\mu dt=\int_{V} \lim\limits_{T\to +\infty} \frac{u(x,T)-u(x,0)}{T} \varphi d\mu=0,
	\end{equation*}
	\begin{equation*}
			\lim\limits_{T\to +\infty}\int_{0}^{T}\int_{V} \frac{\varphi\Delta_{{V}}u }{T} d\mu dt=\int_{V} \lim\limits_{T\to +\infty} \int_{0}^{T}\frac{\Delta_{{V}} u }{T}dt ~\varphi d\mu
			=\int_{V} \Delta_{{V}} v(x)\varphi d\mu,		
	\end{equation*}
	and
	\begin{equation*}
			\lim\limits_{T\to +\infty}\int_{0}^{T}\int_{V} \frac{f(x,u)\varphi }{T} d\mu dt=\int_{V} \lim\limits_{T\to +\infty} \int_{0}^{T}\frac{f(x,u) }{T}dt ~\varphi d\mu
			=\int_{V} f(x,v) \varphi d\mu.		
	\end{equation*}
Letting $T\to +\infty$ in \eqref{5}, we conclude that 
$$
		-\int_{V} \Delta_{{V}} v(x) \varphi(x) d\mu= \int_{V} f(x,v) \varphi d\mu.
$$
It follows that
$
			-\Delta_{{V}} v(x)= f(x,v),~x\in V
$ 
which implies that $\lim\limits_{t\to +\infty}u(x,t)=v(x)$. Then 
$$
	\lim\limits_{t\to +\infty} \max\limits_{x\in V}|u(x,t)-v(x)|=0
$$ 
by the finiteness of $V$.
If we take $e=v$, then \eqref{101} holds.
\end{proof}

\subsection{Applications}
In this subsection, we give some examples as applications of the upper and lower solutions method. 
\subsubsection{The parabolic problem with intial-boundary value} 
In this subsubsection, we consider the discrete Logistic model \eqref{1p1}.

 We first consider the initial boundary value problem
\begin{equation}\label{1p}
	\begin{cases}
		u_{t}-\Delta_{{\Omega}} u= u(a-b u),~&x\in \Omega,~t>0, \\
		u=0,~&x\in \partial{\Omega},~t>0,\\
		u(x,0)=u_0(x),~&x\in \Omega.
	\end{cases}
\end{equation}
Its equilibrium problem is 
\begin{equation}\label{2p}
	\begin{cases}
		-\Delta_{{\Omega}} u= u(a-bu),~x\in \Omega, \\
		u(x)=0,~x\in \partial\Omega. 
	\end{cases}	
\end{equation}

\begin{Theorem}\label{51}
Let $\Omega$ be a subgraph of a host graph $G$.	Suppose $u_0$ $\in R\bar{\Omega}$ satisfying $u_0 \ge, \not \equiv 0$ on $\Omega$ and $a,b>0$. Then \eqref{1p} admits a unique nonnegative global solution $u(x,t)$ and \eqref{2p} admits a unique positive solution $u_{s}(x)$. Furthermore, if $a>\lambda_{1}$, then  
\begin{equation}\label{20220819}
	\lim\limits_{t\to +\infty} u(x,t) = u_{s}(x)~\text{uniformly~for}~x\in \bar{\Omega};
\end{equation}
if $a\le \lambda_{1}$, then 
$
	\lim\limits_{t\to +\infty} u(x,t)=0
$
uniformly for $x\in \Omega$, where $\lambda_{1}$ is the least eigenvalue of the eigenvalue problem \eqref{D}.
\end{Theorem}
\begin{proof}
  It is easy to see that $M:=\max\{\frac{a}{b}, \max\limits_{{\Omega}} u_0 \}$ is a supersolution of \eqref{1p} and $0$ is a subsolution of \eqref{1p}. By Theorem \ref{sxj1}, problem \eqref{1p} admits a unique positive solution $u(x,t)$ satisfying $0\le u(x,t) \le M$ for $x\in \bar{\Omega}$ and $t\ge 0$.
 
Clearly, 
  $\frac{a}{b}$ is a supersolution to \eqref{2p}, and there exists a sufficiently small $\delta^{'}>0$ such that for any $\delta<\delta^{'}$,  $\delta \phi_1$ is a subsolution of \eqref{2p}, where $\phi_1$ is an eigenfunction corresponding of $\lambda_{1}$. We applly Theorem \ref{44} to conclude that \eqref{2p} has a maximal solution $\tilde{u}(x):$ $\delta\phi_1 \le \tilde{u}\le \frac{a}{b}$. 
  
Now, we remain to show that $\tilde{u}$ is the unique positive solution to problem \eqref{2p}. Suppose $v$ is a positive solution of \eqref{2p} satisfying $v>0$ on $\Omega$. It is easy to check that 
 \begin{equation*}
 	\begin{cases}
 		-\Delta_{{\Omega}}(\frac{a}{b}-v)+bv(\frac{a}{b}-v)=0,~x\in \Omega	\\
 		\frac{a}{b}-v(x)\ge 0,~x\in\partial{\Omega}.
 	\end{cases}
 \end{equation*}
 Applying Theorem \ref{25}, we obtain that $\frac{a}{b}\ge v(x)$ on $\bar{\Omega}$. It is easily seen that $\frac{a}{b}$ and $0$ are upper and lower solutions to \eqref{2p} respectively. By Theorem \ref{44}, there exists a maximal solution $\hat{u}$ of \eqref{2p} with $0\le\hat{u}\le \frac{a}{b}$. Thus we see that $v\le\ \hat{u}$ on $\bar{\Omega}$.
 Then integration by parts gives
	\begin{equation*}
		\begin{aligned}
			\int\limits_{\Omega} v\hat{u} (a-b\hat{u}) d\mu&= \int\limits_{\Omega} v(-\Delta_{\Omega} \hat{u}) d\mu 
			=\int\limits_{\bar \Omega} v(-\Delta_{{\Omega}} \hat{u}) d\mu \\
			&= \int\limits_{\bar \Omega} -\hat{u}\Delta_{{\Omega}} v  d\mu 
		=\int_{\Omega} -\hat{u} \Delta_{{\Omega}} v d\mu	=\int\limits_{{\Omega} } \hat{u} v (a-bv) d\mu. 
		\end{aligned}
	\end{equation*}
This implies that
$
	0=\int_{\Omega} v\hat{u} b(\hat{u}-v) d\mu.
$ Due to $\hat{u}\ge v$ in $\Omega$, we conclude that $v\equiv \hat{u}$ on $\Omega$. In other word, the positive solution to \eqref{2p} is unique. 

  Suppose that $a>\lambda_{1}$. We shall prove that \eqref{20220819} holds.
Using Theorem \ref{sm}, we see that $u(x,t)>0$ for $x\in \Omega$, $t>0$. For fix $T>0$, we can find a sufficiently small $0<\epsilon <\delta'$ and a sufficiently large $k>0$ such that 
$
	\epsilon\phi_{1}(x) \le u(x,T) \le k\hat{u} (x),~x\in \bar{\Omega}.
$ 
It is easy to check that $\epsilon\phi_1$ ($k\hat{u}$) is a subsolution (supersolution) to  the following problem
\begin{equation}\label{58}
	\begin{cases}
	v_{t}- \Delta_{{\Omega}} v=v(a-bv)~&\text{in}~\Omega\times (0,T],\\
	v=0~&\text{on}~\partial{\Omega}\times [0,T],\\
	v(x,0)=u(x,T)~&\text{on}~\bar{\Omega}.
	\end{cases}
\end{equation}
 Then it follows from Theorem \ref{sxj1} that the problem \eqref{58} admits a unique solution $v(x,t)=u(x,t+T)$ satisfying $0\le v(x,t)\le M$ for $t\ge 0$, $x\in\bar{\Omega}$. Set $u_1$ and $u_2$ be solutions to \eqref{58} with $v(x,0)=\epsilon\phi_1$ and $v(x,0)=k\hat{u}$, respectively. By Theorem \ref{o}, we deduce that 
$$
	u_{1}(x,t)\le v(x,t) \le u_{2}(x,t)~\text{for}~x\in\bar{\Omega}~\text{and}~t\ge 0.
$$
 Combining Theorem \ref{45} and the fact that {$\hat{u}$} is the unique positive solution to {\eqref{2p}}, we deduce that 
$
	u_{1}(x,t)\to \hat{u}(x)~\text{and}~u_{2}(x,t)\to \hat{u}(x) 
$
uniformly for $x\in \bar{\Omega}$ as $t\to +\infty$. It's easily seen that  $
	u(x,t)=v(x,t-T)\to \hat{u}(x)
$ uniformly for $x\in\bar{\Omega}$ as $t\to +\infty$. Taking $u_s=\hat{u}$, then \eqref{20220819} holds.

Next, we deal with the case that $a\le \lambda_{1}$.
By H$\ddot{\text{o}}$lder inequality, we obtain 
	\begin{equation*}
			\left( \int\limits_{\Omega} u\phi_{1} d\mu\right) ^2 =\left( \int\limits_{\Omega} u \sqrt{{\phi}_{1}}\sqrt{{\phi}_{1}} d\mu\right) ^2 
			\le \int\limits_{\Omega} u^{2} {\phi}_{1} d\mu \int_{\Omega} {\phi}_{1} d\mu 
			\le C_0 \int\limits_{\Omega} u^{2} {\phi}_{1} d\mu,
	\end{equation*}
where $C_0=\text{Vol}(\Omega)\max_{\Omega}\phi_{1} $
Due to $u|_{\partial\Omega}={\phi}_{1}|_{\partial\Omega}=0$, by Proposition \ref{20220803-1} (4), integration by parts gives
\begin{equation*}
	-\int_{\Omega} {\phi}_{1} \Delta_{{\Omega}} u d\mu=-\int_{\bar{\Omega}} {\phi}_{1}\Delta_{{\Omega}} u d\mu=-\int_{\bar{\Omega}} u\Delta_{\Omega} {\phi}_{1} d\mu=-\int_{\Omega}u\Delta_{\Omega} {\phi}_{1} d\mu.
\end{equation*} 
Thus, we see that
	\begin{equation*}
		\begin{aligned}
	\int\limits_{\Omega}	\left(u_{t} {\phi}_{1}+  \lambda_{1} u{\phi}_{1} \right) d\mu=\int_{\Omega} \left(u_{t}{\phi}_{1}-u\Delta_{{\Omega}} {\phi}_{1} \right)d\mu &=	\int\limits_{\Omega}\left(	u_{t} {\phi}_{1}- {\phi}_{1}\Delta_{{\Omega}} u \right) d\mu = \int\limits_{\Omega} \left(au{\phi}_{1}-bu^{2} {\phi}_{1} \right)d\mu\\
			&\le \int\limits_{\Omega} au{\phi}_{1} d\mu -\frac{b}{C_0}(\int\limits_{\Omega} u{\phi}_{1} d\mu)^2.
		\end{aligned}		
	\end{equation*}

Set $f(t)=\int_{\Omega} u(x,t) {\phi}_{1}(x) d\mu$. It's easily seen that
$
	f^{\prime} (t) \le -\frac{b}{C_0} f^{2}(t)~\text{and}~f(t)>0
$
for $t\ge 0$. It follows that 
$$
	f(t) \le \frac{1}{\frac{bt}{C_0} +\frac{1}{f(0)}}.
$$
This implies that $\lim\limits_{t\to +\infty} f(t)=0$, i.e., 
$\int_{\Omega}u(x,t) {\phi}_{1}(x) d\mu\to 0$ as $t\to +\infty$. Since $0\le u \le M$ on $\overline{\Omega} \times [0,+\infty)$, we can find $t_i \to +\infty$ such that
 $$u(x,t_i)\to \tilde{u}(x) \;\text{uniformly for} \;x\in \overline{\Omega}\;\text{ as}\; i \to +\infty.$$ 
 Thus, we have 
$
	\int_{\Omega} \tilde{u}(x){\phi}_{1}(x) d\mu=0.
$
Recalling that $\tilde{u}(x)\ge 0$ for $x\in\overline{\Omega}$, then we have $\tilde{u}(x)\equiv 0$ on $\overline{\Omega}$. Therefore, we conclude that $u(x,t)\to 0$ uniformly for $x\in \overline{\Omega}$ as $t\to +\infty$. 
\end{proof}

Next we consider the example 
\begin{equation}\label{1n}
	\begin{cases}
		u_{t}-\Delta_{{\Omega}} u= u(a-b u),~&x\in \Omega,~t>0, \\
		\frac{\partial u}{\partial_{\Omega}n}(x,t)=0,~&x\in \partial{\Omega},~t>0,\\
		u(x,0)=u_0(x),~&x\in \Omega.
	\end{cases}
\end{equation}
\begin{Theorem}
Let $\Omega$ be a subgraph of a host graph $G$,	$a>0$ and $b>0$ be constants and $u_0\ge ,\not\equiv 0$ belong to $R\bar{\Omega}$. Then the problem \eqref{1n} has a unique nonnegative global solution $u$, and 
	$
	\lim\limits_{t\to +\infty} u(x,t)=\frac{a}{b}
	$ uniformly for $x\in\bar{\Omega}$. 
\end{Theorem}
\begin{proof}
	Let $M:=\max\{\frac{a}{b}, \max\limits_{\Omega} u_0 \}$.  It is easy to check that $M$ is a supersolution and $0$ is a lower solution of \eqref{1n}. By the upper and lower solutions method (Theorem \ref{sxj1}), we deduce that \eqref{1n} admits a unique solution $u(x,t)\in C^{0}({\bar\Omega }\times[0,\infty) ) \cap C^{1}({\Omega }\times[0,\infty) )$ satisfying $0\le u(x,t)\le M$ for $x\in\bar{\Omega}$ and $t\in [0,+\infty)$.  
	
	It is well known that, for any given $Z_0>0$, the initial value problem 
	$
	Z^{'}(t)=Z(a-bZ),~Z(0)=Z_{0}
	$
	admits a unique solution $Z(t;Z_0)$ for $t\ge 0$, which satisfies $Z(t)\to \frac{a}{b}$ as $t\to +\infty$. Using Theorem \ref{sm}, we have $u(x,t)>0$, $x\in\bar{\Omega
	}$, $t>0$. For fixed $\delta>0$, define $M=\max\limits_{x\in\bar{\Omega}} u(x,\delta)$, $m=\min\limits_{x\in\bar{\Omega}} u(x,\delta)$, we conclude that
	$$
	Z(t;m)\le u(x,t+\delta) \le Z(t;M),~x\in\bar{\Omega},~t\ge 0
	$$
	by Theorem \ref{o}.
It follows that 
	$
	\lim\limits_{t\to+\infty} u(x,t)=\frac{a}{b} 
	$ uniformly for $x\in \bar{\Omega}$.
\end{proof}
\subsubsection{The parabolic equation with intial value}
In this subsubsection, we consider the discrete model \eqref{52}.

 We consider first the Cauchy problem
\eqref{1pb} with $f$ satisfies \eqref{11} and \eqref{52}. Its equilibrium problem is 
\begin{equation}\label{2pb}
		-\Delta_{{V}} u= f(u),~x\in V
\end{equation}

\begin{Theorem}\label{51b}
Let $G=G(V,E)$ be a graph. Suppose that $u_0\in RV$ satisfying $u_0\ge 0,\not \equiv 0$ on $V$,  $f$ satisfies \eqref{11}, \eqref{52} and 
	\begin{equation}\label{523}
	\frac{f(u)}{u}~\text{is monotone nonincreasing on}~[0,1].
	\end{equation}
	 Then \eqref{1pb} admits a unique  nonnegative global solution $u(x,t)$, \eqref{2pb} admits a unique positive solution $u_{s}(x)\equiv 1$ and
	\begin{equation}
		\lim\limits_{t\to +\infty} u(x,t) = 1~\text{uniformly~for}~x\in V;
	\end{equation}
\end{Theorem}
\begin{proof}
 Denote $M:=\max\{1, \max\limits_{V} u_0 \}$.	It is easy to see that $M$ is a supersolution of \eqref{1pb}, $0$ is a subsolution of \eqref{1pb}. By Theorem \ref{sxc},  problem \eqref{1pb} admits a unique solution $u(x,t)$ satisfying $0\le u(x,t) \le M$ for $x\in V$ and $t\ge 0$. 
 
 Noting that $\mu_{0}$ is the smallest eigenvalue of the eigenvalue problem $-\Delta_{V} u=\mu u$ on $V$, and $\Psi_0 $ is the eigenfunction corresponding to $\mu_0$.

 It is easy to see that $1$ is a supersolution to \eqref{2pb} and there exists a sufficiently small $\delta^{'}>0$ such that { $\delta \Psi_0$ is a subsolution of \eqref{2pb} }for all $\delta<\delta^{'}$. We apply Theorem \ref{44} to conclude that \eqref{2pb} has a maximal solution $\hat{u}(x)$ satisfying $\delta\Psi_0 \le \hat{u}\le 1$. 
 
 Next, we show that $\hat{u}$ is the unique positive solution to problem \eqref{2pb}. Suppose that $v$ is a positive solution of \eqref{2pb}. It is easy to check that $v$ satisfies 
 \begin{equation*}
	-\Delta_{{V}}(1-v)+\frac{f(v)}{1-v}(1-v)=0~\text{on~}V.
\end{equation*} 
Thanks to $f^{'}(1)\le 0$, we see that $\frac{f(v)}{1-v}\ge 0$ on $V$. It follows from {Theorem \ref{27}} that $1-v\ge 0$ on $V$. It is easy to see that $1$ is a supersolution of \eqref{2pb} and $0$ is a subsolution of \eqref{2pb}. By Theorem \ref{44c}, there exists a maximal solution $u_s$ to \eqref{2pb} satisfying $0\le u_s \le 1$. Thus, we obtain 
\begin{equation}\label{zx}
	v\le u_s \text{ on~} V.
\end{equation}

Then integration by parts gives
	\begin{equation*}
		\begin{aligned}
			\int\limits_{V} v f({u_s}) d\mu&= \int\limits_{V} v(-\Delta_{V} {u_s}) d\mu			 
			= \int\limits_{ V} -{u_s}\Delta_{{V}} v  d\mu 
			=\int\limits_{{V} } {u_s} f(v) d\mu. 
		\end{aligned}
	\end{equation*}
	Thus, we deduce that 
	\begin{equation*}
		0=\int_{V} vf({u_s})- {u_s}f(v) d\mu=\int_{V} v{u_s}\left(\frac{f({u_s})}{u_s}- \frac{f(v)}{v}   \right) d\mu.
	\end{equation*}
 By \eqref{11} and \eqref{zx}, we deduce that 
$$
		\int_{V} v{u_s}\left(\frac{f({u_s})}{{u_s}}- \frac{f(v)}{v}   \right) d\mu \le 0.
$$
	This implies that $v  \equiv u_{s}$ on $V$. Thus we deduce that the positive solution to \eqref{2pb} is unique. We observe that $1$ is a positive solution to \eqref{2pb}, then $u_{s}\equiv 1$. 
	
 Using Theorem \ref{smc}, we see that $u(x,t)>0$ for $x\in V$, $t>0$. For fixed $T$,  we can find a sufficiently small $0<\epsilon <1$ and a sufficiently large $k>1$ such that 
$
		\epsilon\phi(x) \le u(x,T) \le ku_{s} (x),~x\in V.
$ 
We now consider the following problem
	\begin{equation}\label{58c}
		\begin{cases}
			v_{t}- \Delta_{V} v=f(v)~&\text{in}~V\times (0,+\infty),\\
			v(x,0)=u(x,T)~&\text{on}~V.
		\end{cases}
	\end{equation}
It is easy to check that $\epsilon\phi$ is a subsolution of \eqref{58c} and $ku_{s}$ is a supersolution of \eqref{58c}. By Theorem \ref{sxc}, the problem \eqref{58c} admits a unique solution $v(x,t)=u(x,t+T)$ satisfying $0\le v(x,t)\le M$ for $t\ge 0$, $x\in V$. Let $u_1$ and $u_2$ be solutions to \eqref{58c} with $v(x,0)=\epsilon\phi$ and $v(x,0)=ku_{s}$, respectively. We apply Theorem \ref{oc} to deduce that 
$$
		u_{1}(x,t)\le v(x,t) \le u_{2}(x,t)~\text{for}~x\in V~\text{and}~t\ge 0.
$$
	By Theorem \ref{45c} and the fact that $u_{s}$ is the unique positive solution of \eqref{2pb}, we deduce that 
$
		u_{1}(x,t)\to u_{s}(x)~\text{and}~u_{2}(x,t)\to u_{s}(x) 
$
	uniformly for $x\in V$ as $t\to +\infty$. Thus, we see that $
		u(x,t)=v(x,t-T)\to u_{s}(x)
	$ uniformly for $x\in V$ as $t\to +\infty$.
\end{proof}

Next, we consider the Cauchy problem on graphs with $f$ satisfies \eqref{11} and \eqref{b}.

We begin with the following elementary lemma.
\begin{Lemma}\label{55}
Let $G=G(V,E)$ be a graph. Let $\gamma>0$ be a constant. Suppose that $f\in C^{1}[0,\gamma]$ and $f(0)=0$. Let $u(x,t)$ be a solution to $u_{t}-\Delta_{{V}} u= f(u)$ on $V$ and let $f(u)<0$ in $(0,\gamma]$. If $u(x,0)\in [0,\gamma]$, then $\lim\limits_{t\to \infty} u(x,t)=0 $ uniformly on $V$.
\end{Lemma}
\begin{proof}
	Let $v$ be the solution of the problem $
			v_{t}=f(v),
	v(x,0)=\gamma.$	
Then $t=\int_{v}^{\gamma} -\frac{1}{f(u)} du $. It's easy to see that  $v\to 0$ as $t\to+\infty$. Due to $u(x,0)\ge 0$, by Theorem \ref{oc}, we see that $u(x,t)\ge 0 $ on $V\times [0,+\infty)$. Thanks to $v(x,0)\ge u(x,0)$, we use Theorem \ref{oc}  again to conclude that $v(x,t)\ge u(x,t)$ on $V\times [0,+\infty)$, and hence that $u(x,t)\to 0$ as  $t\to +\infty$ uniformly for $x\in V$.
\end{proof}

For any $\rho\in [0,\alpha)$, we define 
\begin{equation}\label{54}
	s(\rho):=\sup\{\frac{f(u)}{u-\rho}:u\in (\alpha,1) \} 
\end{equation}
and denote $[m]^{+}:=\max\{m,0\}$.

\begin{Theorem}
Let $G=G(V,E)$ be a graph.	Assume that $f(u)$ satisfies \eqref{11} and \eqref{b}. Then the problem \eqref{1pb} admits a unique nonnegative global solution $u(x,t)$. Furthermore, if $u_0$ satisfies 
	\begin{equation}\label{2c}
		e^{\frac{1}{2}}[u_0-\rho]^{+} \mu(x)< \alpha-\rho~\text{ for all}~x\in V, \text{and~ some}~ \rho\in [0,\alpha).
	\end{equation} 
then $
\lim\limits_{t\to +\infty} u(x,t)=0~\text{uniformly for}~x\in V;
$
if $u_0$ satisfies {$u_0> \alpha$} on $V$, then \begin{equation}\label{zcz}
	\lim\limits_{t\to +\infty} u(x,t)=1~\text{uniformly for}~x\in V.
\end{equation}
\end{Theorem}

\begin{proof}
Let $M:=\max\{1, \max\limits_{V} u_0 \}$. It is easily seen that $M$ is an upper solution and $0$ is a lower solution of of \eqref{1pb}. We apply Theorem \ref{sxc} to deduce that \eqref{1pb} admits a unique nonnegative global solution $u(x,t):$ $0\le u \le M$ for $x\in V$ and $t\ge 0$.	Fix $\rho\in [0,\alpha)$. By the proof of Theorem \ref{smc}, the problem
	\begin{equation*}
		\begin{cases}
			w_{t}-\Delta_{{V}} w=sw, ~x\in V,~t>0,\\
			w(x,0)=[u(x,0)-\rho]^{+},~x\in V
		\end{cases}
	\end{equation*}
admits a unique solution 
\begin{equation}\label{3c}
	w(x,t)=e^{st} \sum\limits_{y \in V}p(x,y,t)[u(y,0)-\rho]^{+} \mu(y).
\end{equation} 
It follows from Theorem \ref{smc} that $w\ge 0$ on $V\times[0,+\infty)$ and hence that $w\equiv [w]^{+}$. In view of $f(u)\le 0$ on $[0,\alpha]$, by \eqref{54}, we deduce that $f(u)\le s[u-\rho]^{+}$. 

Setting $v(x,t):=u(x,t)-\rho$. Then 
$$
v_{t}-\Delta_{{V}} v-s[v]^{+}\le u_{t}-\Delta_{{V}} u -f(u)=0=w_{t}-\Delta_{{V}} w-s[w]^{+}.
$$
Using Theorem \ref{smc}, we have $v(x,t)\le w(x,t)$ so that $u(x,t)\le w(x,t)+\rho$. By \eqref{2c} and \eqref{3c}, we conclude that 
\begin{equation*}
	\begin{aligned}
		u(x,\frac{1}{2s}) &\le e^{\frac{1}{2}} \sum\limits_{y\in V} p(\frac{1}{2s},x,y)[u(y,0)-\rho]^{+}\mu(y)+ \rho\\
		&=e^{\frac{1}{2}} \sum\limits_{y\in V} \sum\limits_{j=0}^{N-1} e^{-\mu_{j}\frac{1}{2s}}\Psi_{j}(x)\Psi_{j}(y)\mu(y)[u(y,0)-\rho]^{+}+\rho \\
		&\le e^{\frac{1}{2}}[u(x,0)-\rho]^{+}\mu(x) +\rho \\
		&< \alpha  .
	\end{aligned}	
\end{equation*}
Then by Lemma \ref{55}, we obtain $u(x,t)\to 0$ uniformly for $x\in V$.

Suppose that $u_0(x)> \alpha$ on $V$. We use \eqref{1pb} to obtain that $u(x,t)$ satisfies
$$
		u_{t}-\Delta_{{V}} u-\frac{f(u)}{u} u=0,~x\in V,~t>0, 
		u(x,0)=u_0(x)\ge 0,~x\in V,
$$ 
Due to $0\le u\le M$, by \eqref{b}, for any given $T>0$, $\frac{f(u)}{u}$ is bounded for $x\in V$ and $t\in [0,T]$. Thus, by Theorem \ref{smc}, $u(x,t)>0$ for $x\in V$ and $t>0$. Choosing $m,M\in\mathbb{R}$ such that {$\alpha<m<u_0<M$} on $V$. It is well-known that the initial problem 
$$	
		v_{t}=f(v),
		v(0)=q
$$
admits a unique solution $v(t;q)$ satisfying 
$
	\lim\limits_{t\to +\infty} v(t;q)=1,
$
where $q=m$ or $M$.
We use Theorem \ref{oc} to conclude that $v(t;m)\le u(x,t)\le v(t;M)$ for $x\in V$ and $t>0$. Therefore, we obtain \eqref{zcz}.
\end{proof}

Finally, we give an example and numerical experiments to demonstrate Theorem \ref{51}.

\begin{example}\label{61}
	Choose a graph $\overline{\Omega}$ satisfying $\Omega=\{x_1 ,x_2, x_3 \}$ and $\partial{\Omega}=\{x_4, x_5 \}$ whose vertices are linked as the following figure with a weight $\omega$ satisfying 
$
		\omega_{x_i x_j}=\begin{cases}
			1,~x_i \sim x_j,\\
			0,~x_i \not= x_j
		\end{cases}
$
for $i,j=1,2,\cdots,5$. 

\begin{center}
\begin{tikzpicture}[>=stealth, scale=1.2]
	
	\draw(0,0)node{\hspace{-.1cm}$\circ$}--(1,0)node{$\bullet$}--(2,0)node{$\bullet$}--(3,0)node{\hspace{.1cm}$\circ$};
	\draw(1,0)--(1.5,0.867)node{$\bullet$}--(2,0);
	\draw(0,0)node[below]{$x_4$}(1,0)node[below]{\hspace{-.1cm}$x_1$}(2,0)node[below]{\hspace{.3cm}$x_3$}(3,0)node[above]{$x_5$} (1.5,0.867)node[above]{$x_2$};
\end{tikzpicture}
\end{center}
Suppose that $u(x,t)$ is the unique global positive solution of \eqref{1p} with $b=1$, $u_{0}(x_1)=8$, $u_{0}(x_2)=1$, $u_{0}(x_3)=0.5$ and $\mu(x)=\sum\limits_{y\sim x} w_{xy}$ for $x\in \bar{\Omega}$. 
  Next, we calculate the principle eigenvalue $\lambda_{1}$ of $-\Delta_{\Omega}$ on $\overline{\Omega}$. 
  We need solve the following eigenvalue problem
\begin{equation}\label{1}
		-\Delta_{\Omega}\phi=\lambda \phi,~x\in \Omega,~\phi(x)=0,~x\in \partial\Omega.
\end{equation}

It is easy to check that the eigenvalues of the problem \eqref{1} are 
\begin{equation*}
	\lambda_{1}=\frac{5-\sqrt{13}}{6} ,~\lambda_{2}=\frac{5+\sqrt{13}}{6},~\lambda_{3}=\frac{4}{3}.
\end{equation*}

If we choose $a=0.1$, then $a<\lambda_{1}$. Therefore, by Theorem \ref{51}, we see that 
\begin{equation*}
	\lim\limits_{t\to +\infty} u(x,t)=0
\end{equation*}
uniformly for $x\in \bar{\Omega}$. The numerical experiment result is shown in Figure \ref{fig1} (a).
\end{example}

Besides, if we choose $a=1.8$, then $a>\lambda_{1}$. Thus, by Theorem \ref{51}, we deduce that 
\begin{equation*}
	\lim\limits_{t\to +\infty} u(x,t) = u_{s}(x)~\text{uniformly~for}~x\in \bar{\Omega},
\end{equation*}
where $u_{s}(x)$ is the unique positive solution to 
	\begin{equation*}\label{}
	\begin{cases}
		- \Delta_{{\Omega}} u=u(1.8-u)~&\text{in}~\Omega,\\
		u=0~&\text{on}~{\Omega}.
	\end{cases}
\end{equation*}
The numerical experiment result is shown in Figure \ref{fig1}  (b).

\begin{figure*}[!t]
	\centering
	\subfigure[Extinction] {\includegraphics[height=2in,width=3in,angle=0]{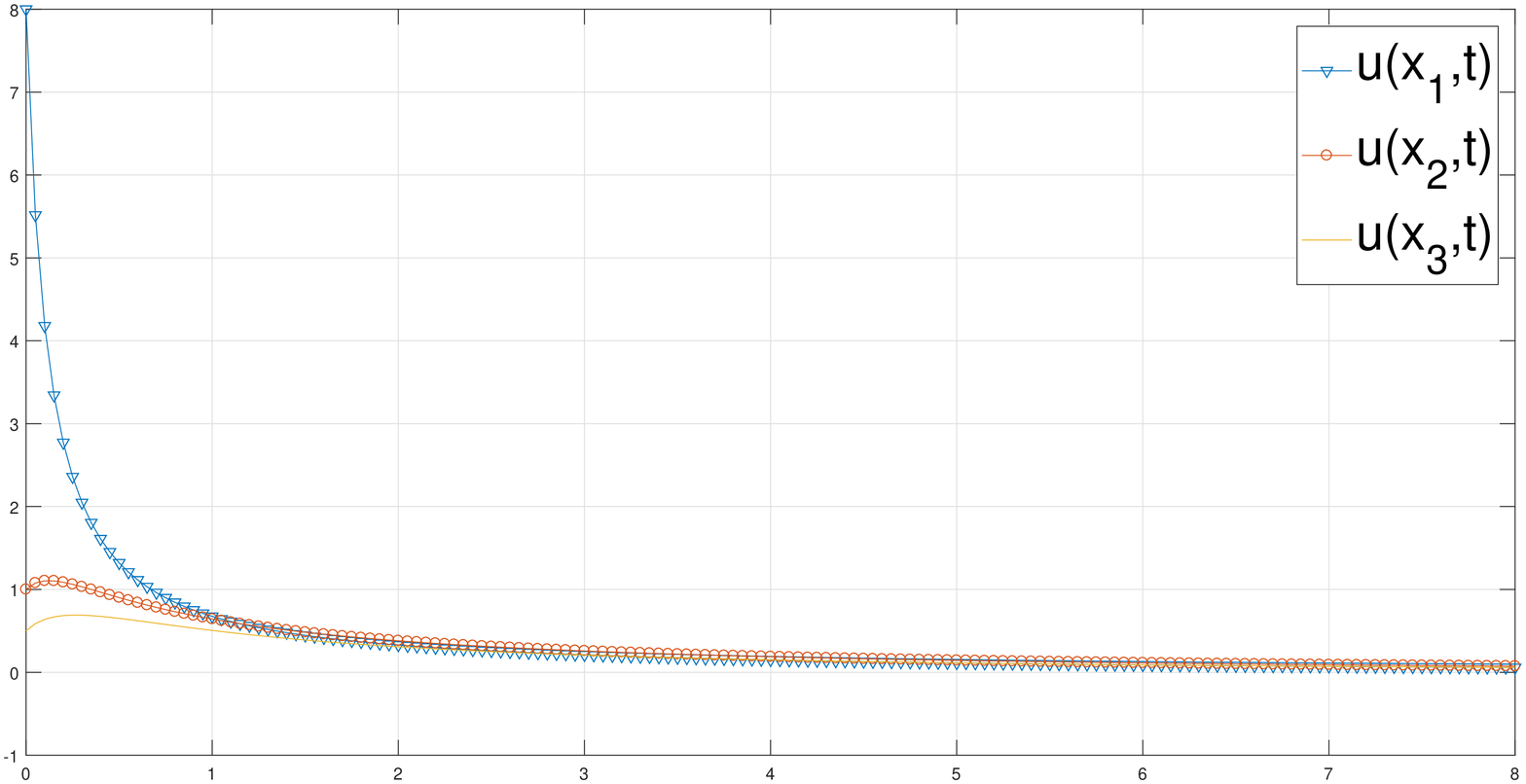}}
	\subfigure[Establishment] {\includegraphics[height=2in,width=3in,angle=0]{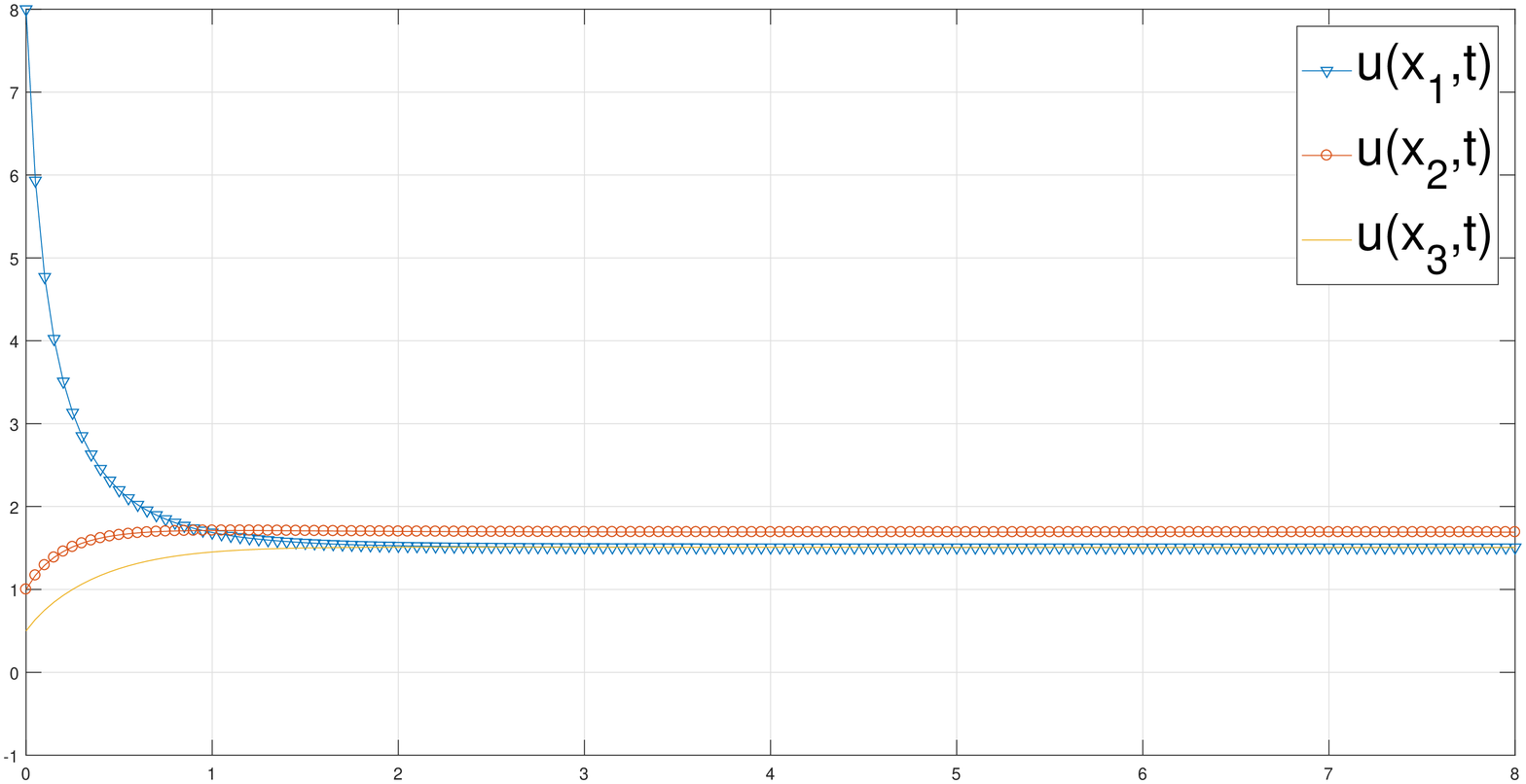}}
	\caption{}
	\label{fig1}
\end{figure*}

\begin{proof}[Acknowledgements.]\renewcommand{\qedsymbol}{}	
		The authors would like to thank the anonymous Referees for their valuable comments
	which helped to improve the manuscript.
\end{proof}

\end{document}